\documentclass[12pt]{amsart}
\usepackage{amsmath, amsthm, amssymb,cite,enumitem}
\usepackage{fullpage}
\usepackage{color}
\usepackage{hyperref}
\usepackage[makeroom]{cancel}
\usepackage[title,titletoc]{appendix}
\usepackage{amsmath}
\usepackage{comment}
\usepackage{amssymb}
\usepackage{amsrefs}
\usepackage{enumitem}
\usepackage{mathtools}
\usepackage{mathrsfs}
\newcommand{\wh}[1]{\widehat{#1}}

\newcommand{\dd}{d}

\newcommand{\dH}{\dot{H}}

\newcommand\R{\mathbb{R}}

\newcommand\N{\mathbb{N}}

\newcommand{\dx}{\,\dd x}

\newcommand{\dddt}{\,\frac{\dd}{\dd t}}
\newcommand{\eps}{\varepsilon}

\newcommand{\snorm}[1]{\lVert#1\rVert}

\newtheorem{theorem}{Theorem}[section]
\newtheorem{lemma}[theorem]{Lemma}
\newtheorem{corollary}[theorem]{Corollary}
\newtheorem{proposition}[theorem]{Proposition}

\theoremstyle{definition}

\numberwithin{equation}{section}
\begin{document}

\title{Well--posedness for the continuum Calogero--Moser equations by parabolic regularization}

\author[Z.~Yingmo]{Yingmo Zhang}
\address{School of Mathematics, Georgia Institute of Technology. 686 Cherry Street, Atlanta, GA}
\email{yzhang4029@gatech.edu}

\begin{abstract}

We prove that the focusing and defocusing continuum Calogero--Moser equations are globally well-posed in $H^s_+(\mathbb R)$ for every integer $s\geq 3$. In the focusing case, this requires the initial data to satisfy the mass condition
$$\|q_0\|_{L^2}^2<2\pi.$$
Our proof is based on a parabolic regularization of the equation, uniform \textit{a--priori} estimates for the regularized solutions, and a Bona--Smith type approximation argument.
\end{abstract}

\maketitle
\tableofcontents
%%%%%%%%%%%%%%%%%%%%%%%%%%%%%%%%%%%%%%%%%%%%%%%%%%%%%%%%%%%%%%%%%%%%%%%%%%%%%%%%%%%%%%%%%%%%%%%%%%%%%%%%

\section{Introduction}
We study the global well--posedness problem for the following two dispersive equations:
\begin{equation}\tag{CCM}\label{CCM}
i\tfrac{d}{dt} q= -q'' \pm 2iq C_+\big( |q|^2\big)' .
\end{equation}
Here and throughout the paper, the upper sign corresponds to the focusing equation, while the lower sign corresponds to the defocusing equation. The unknown function 
$$q=q(t,x):\mathbb{R}\times\mathbb{R}\longrightarrow\mathbb{C}$$
is a complex-valued function, and the prime denotes differentiation with respect to the spatial variable $x$. Also, $C_+$ denotes the Cauchy--Szeg\H{o} projection onto nonnegative Fourier frequencies, defined by
$$\widehat{C_+f}(\xi)=\mathbf{1}_{[0,\infty)}(\xi)\widehat f(\xi).$$

The solutions are sought in the Hardy--Sobolev space
$$
H^s_+(\R):=\left\{f\in H^s(\mathbb{R}):\hat{f}(\xi)=0 \text{ for }\xi<0\right\}.
$$
In particular, $H^0_+(\mathbb R)=L^2_+(\mathbb R)$. The nonnegative-frequency condition is preserved by the evolution, so $L^2_+(\mathbb R)$, and more generally $H^s_+(\mathbb R)$, form natural phase spaces for \eqref{CCM}.

The focusing and defocusing equations in \eqref{CCM} arise from two different settings. The focusing equation was formally derived in \cite{AbanovBettelheimWiegmann} as a continuum limit of the classical Calogero--Moser particle system
$$
\frac{d^2x_j}{dt^2}
=
\sum_{k\neq j}
\frac{1}{(x_j-x_k)^3}.
$$
This finite-dimensional completely integrable system describes a collection of particles interacting through an inverse-square potential. The connection with this particle model accounts for the terminology continuum Calogero--Moser equation, and is also closely related to the completely integrable structure of the focusing CCM equation.

The defocusing equation has its origin in the theory of internal waves. More precisely, it first appeared as a special case of the intermediate nonlinear Schrödinger equation introduced by Pelinovsky in \cite{Pelinovsky}. The latter is a modulation equation describing the slowly varying envelope of an approximately monochromatic internal wave propagating along the interface of a stratified two-fluid system. In its finite-depth form, the equation can be written as
$$
\partial_t u+i\partial_x^2u
=
\beta u(1+iT_h)\partial_x\bigl(|u|^2\bigr)
+i\gamma |u|^2u,
$$
where $h>0$ is the depth parameter and $T_h$ is the Fourier multiplier defined by
$$
\widehat{T_hf}(\xi)
=
-i\coth(h\xi)\widehat f(\xi).
$$
As $h\to\infty$, one formally has
$$
-i\coth(h\xi)
\longrightarrow
-i\operatorname{sgn}(\xi),
$$
so that $T_h$ converges to the Hilbert transform $H$. Since
$$
1+iH=2C_+
$$
on nonzero frequencies, the infinite-depth limit of the intermediate nonlinear Schrödinger equation, with $\gamma=0$, reduces to the defocusing CCM equation, up to the sign and normalization conventions used in \eqref{CCM}. Thus, the finite-depth intermediate nonlinear Schrödinger equation is a more general model, while the defocusing CCM equation corresponds to its infinite-depth specialization.

Our goal is to construct global solutions to \eqref{CCM} and to establish their uniqueness and continuous dependence on the initial data. Our main result is the following.
\begin{theorem}\label{main_thm}
Let $s\geq 3$ be an integer and let $q_0\in H^s_+(\mathbb R)$. 
In the focusing case, assume additionally that
\begin{equation}\label{aug7_0950}
    \|q_0\|_{L^2}^2<2\pi.
\end{equation}
Then the initial-value problem for \eqref{CCM} with $q(0)=q_0$ admits a unique global solution
$$q\in C\bigl([0,\infty);H^s_+(\mathbb R)\bigr)$$
with the compact--open topology. Moreover, for every time $T>0$, the data-to-solution map
$$q_0\longmapsto q$$
is continuous from $H^s_+(\mathbb R)$ into 
$C\bigl([0,T];H^s_+(\mathbb R)\bigr) \cap C^1\bigl([0,T];H^{s-2}_+(\mathbb R)\bigr)$.
\end{theorem}

We next review some known results concerning \eqref{CCM}. In \cite{GérardLenzmann2024}, Gérard and Lenzmann proved global well-posedness for the focusing equation in $H^s_+(\mathbb R)$ for integers $s\geq 1$ under the mass condition \eqref{aug7_0950}.
They also classified the traveling waves and studied multisoliton solutions. In \cite{KLV}, Killip, Laurens, and Vi\c{s}an established global well-posedness in the scaling-critical space $L^2_+(\mathbb R)$ for the defocusing equation, and for the focusing equation under the condition \eqref{aug7_0950}.  The periodic analogue was studied by Badreddine in \cite{BadreddineGWP}, where global well-posedness was established under the name of the Calogero--Sutherland derivative nonlinear Schr\"odinger equation. 

Very recently, Chapouto, Forlano, and Laurens \cite{CFL} proved well-posedness for the intermediate nonlinear Schr\"odinger equation in $H^s(\mathbb R)$ for every $s>0$, covering the full scaling-subcritical range. For prior results, see \cite{Hadama, deMoura, deMouraPilod, BarrosdeMouraSantos,CFL2025}

For the focusing equation, Hogan and Kowalski \cite{HoganKowalski} studied the threshold for the growth of $H^s$ norms for $s>0$. Kim, Kim, and Kwon \cite{KimKimKwon} proved that there are solutions in $H^s_+(\mathbb{R})$ for every $s>0$ that blow--up in finite time and further shows that the restriction \eqref{aug7_0950} is necessary. Also, Kim and Kwon \cite{KimKwon} proved soliton resolution for both finite-time blow-up solutions and global-in-time solutions. Further results on the construction and classification of blow-up solutions were obtained in \cite{JeongKim,JeongKimKimKwon}. For the defocusing equation with a nonvanishing condition at infinity, Chen \cite{ChenNonvanishing} proved global well-posedness in a Zhidkov space.

Although sharper results are already known, the purpose of the present paper is to address the well-posedness problem for \eqref{CCM} at higher regularity. To our knowledge, parabolic regularization has not previously been applied to the CCM equations. Such an approach is of interest because it relies mainly on energy estimates and compactness arguments, and is therefore more robust. Moreover, the same type of parabolic regularization can also be formulated on $\mathbb T$, where the heat semi--group has analogous smoothing properties; see, for example, \cite[Rem.~3.5, p.~58]{ErdoganTzirakis2016}.

Our argument is inspired by the approximation method of Bona and Smith for the KdV equation \cite{BonaSmith}. However, its adaptation to \eqref{CCM} is not immediate. In the KdV setting, the uniform a priori estimates are obtained successively at the integer regularities $L^2$, $H^1$, $H^2$, and higher Sobolev spaces. This procedure relies on the infinite sequence of polynomial conserved quantities of KdV, whose leading terms control successively higher integer-order Sobolev norms. For \eqref{CCM}, after the $L^2$-norm, the next polynomial conserved quantity controls the fractional norm $\dot H^{\frac{1}{2}}$. In this step, we have the energy
$$E_1:=\snorm{q}^2_{\dH^\frac12}\mp \snorm{C_+(|q|^2)}^2_{L^2}$$
following the definition of $E_n$ in Proposition \ref{jul29_1601}. In the focusing case, we rely on a sharp inequality from \cite{GérardLenzmann2024} (see Lemma \ref{GL_type_original} below for details) in order to obtain coercivity precisely below the sharp mass threshold
$$
\|q_0\|_{L^2}^2<2\pi.
$$
The $\dot H^{\frac{1}{2}}$ estimate is coupled to an $H^2$ estimate, and the two bounds must be closed simultaneously by a bootstrap argument. Only after obtaining the uniform $H^2$ bound can we derive the higher $H^s$ estimates for integers $s\geq 3$. Much like \eqref{CCM}, the polynomial conserved quantities contain the nonlocal operator $C_+$, which prevents commuting factors of $q$ when integrating by parts.  To circumvent this, we develop a substitute in Lemma \ref{Commutator_type} in order to ensure that derivatives do not fall on the least regular factor.

Our proof follows the approximation argument introduced by Bona and Smith in \cite{BonaSmith}. Given $q_0\in H^s_+(\mathbb R)$, we regularize the initial data by setting
$$
q_{0,\varepsilon}:=P_{\leq \varepsilon^{-\frac{1}{3}}}q_0
$$
and consider the parabolically regularized equation
\begin{equation}\label{CCM2}\tag{$\text{CCM}_\eps$}
i\partial_tq_\varepsilon-i\varepsilon\partial_x^2q_\varepsilon
=
-\partial_x^2q_\varepsilon
\pm 2iq_\varepsilon C_+\partial_x\left(|q_\varepsilon|^2\right).
\end{equation}
The smoothing effect of the corresponding linear propagator allows us to establish local well-posedness for the regularized equation by a contraction-mapping argument.  
As described above, we then derive a priori estimates that are uniform in $\varepsilon$ on every compact time interval. %We first obtain $L^2$ and $\dot H^{\frac{1}{2}}$ estimates. The $\dot H^{\frac{1}{2}}$ estimate is coupled to an $H^2$ estimate, and the two bounds are closed by a bootstrap argument. Higher-order energy estimates then give uniform control in $H^s_+(\mathbb R)$ for every integer $s\geq3$. These estimates extend the regularized solutions globally.

Instead of parabolic regularization, the original argument of Bona and Smith
\cite{BonaSmith} relied on a BBM--type regularization. The analogous
regularization for \eqref{CCM} would be
\begin{equation}\label{bonasmithonccm}
    i\partial_t q
    -i\varepsilon\partial_x^2\partial_t q
    =
    -\partial_x^2q
    \pm 2iqC_+\partial_x\left(|q|^2\right).
\end{equation}
For the KdV equation, Bona and Smith \cite{BonaSmith} obtain uniform a--priori bounds by exploiting the first few conservation laws of KdV, suitably adapted to the regularized equation. For \eqref{bonasmithonccm}, however, the analogous energy argument encounters a difficulty already at the level of the first nontrivial CCM energy. More precisely, when differentiating the nonlinear part, one encounters the contribution
$$
\frac{d}{dt}(E_1+\eps\int|q''|^2 dx)=2\varepsilon\operatorname{Re}
\int
C_-\left(|q|^2\right)
C_+\left(
\overline q\,\partial_x^2\partial_tq
+
q\,\partial_x^2\partial_t\overline q
\right)\,dx.
$$
If one attempts to eliminate the time derivative by means of
\eqref{bonasmithonccm}, factors involving $\partial_x^4q$ arise. Consequently,
instead of coupling the $\dot H^{\frac12}$ estimate with an $H^2$ estimate, as in
our parabolic regularization, one would need to couple the $\dot H^{\frac12}$
estimate with an $H^4$ estimate.

This difficulty is avoided by the parabolic regularization \eqref{CCM2}. In this case, differentiating the first energy produces instead the error term
$$ \frac{d}{dt}E_1 =\varepsilon \int |q|^2  \left( \overline q\,q'' + q\,\overline q'' \right)\,dx, $$
as in \eqref{aug7_1613}, which involves only two spatial derivatives and can therefore be controlled using the $H^2$ estimate developed below.

Next, we compare two regularized solutions $q_\varepsilon$ and $q_\delta$, where $0<\delta\leq\varepsilon$, and prove that
$\left\{q_\varepsilon\right\}_{\varepsilon>0}$
is a Cauchy family in the $C\left([0,T];H^s_+(\mathbb R)\right)$ topology for every $T>0$. In contrast to the a--priori estimates, the convergence argument does not require a coupled bootstrap between different Sobolev norms. We first establish an $L^2$ estimate for the difference $q_\varepsilon-q_\delta$. This smallness is then combined with the uniform a--priori bounds and interpolation to control the higher order difference estimates directly at the integer Sobolev levels. The limit is shown to solve \eqref{CCM}. Uniqueness follows from an $L^2$ energy estimate for the difference of two solutions and Gr\"onwall's inequality.

Finally, continuous dependence on the initial data follows from the standard $\frac{\varepsilon}{3}$ argument from \cite{BonaSmith}.

The remainder of the paper is organized as follows. In Section 2, we introduce the notation and collect the preliminary estimates used throughout the paper. In Section 3, we establish local well-posedness for the regularized equation by a contraction-mapping argument. Section 4 is devoted to the a-priori estimates that are uniform in the regularization parameter and to the global existence of the regularized solutions. In Section 5, we prove that the regularized solutions form a Cauchy family, pass to the limit as $\varepsilon\to0$, and establish existence and uniqueness for \eqref{CCM}. In Section 6, we prove the continuous dependence of the solutions on the initial data.

\subsection*{Acknowledgements}
The author would like to express sincere gratitude to Dr. Thierry Laurens for his generous guidance and for the many insightful discussions throughout this project.
%%%%%%%%%%%%%%%%%%%%%%%%%%%%%%%%%%%%%%%%%%%%%%%%%%%%%%%%%%%%%%%%%%%%%%%%%%%%%%%%%%%%%%%%%%%%%%%%%%%%%%%%
\section{Preliminaries}\label{preliminaries}
Throughout the paper, we use the Fourier transform convention
 $$
    \wh{f}(\xi) = \frac{1}{\sqrt{2\pi}} \int e^{-ix\xi} f(x)\,dx ,
$$
 With this choice, we have
    \[
    \wh{fg}(\xi) = \frac{1}{\sqrt{2\pi}} [\wh{f}*\wh{g}](\xi)
    \quad\text{and}\quad
    \langle f,g \rangle = \langle \wh{f} , \wh{g} \rangle .
    \]
We also denote $\langle f,g\rangle:=\int \overline{f}g\,dx $ the $L^2$--inner product of the functions $f,g$ on the line.

In this section, we collect two estimates associated with the Hardy-space structure of the Cauchy--Szeg\H{o} projection. The first is a sharp product estimate due to G\'erard and Lenzmann \cite{GérardLenzmann2024}, which will be used to obtain the coercivity of the first \eqref{CCM} energy below the mass threshold $2\pi$. 
\begin{lemma}\label{GL_type_original}
    For all $f\in L^2_+$ and $g\in \dH^{\frac{1}{2}}$, we have
        \begin{equation}\label{GL_type_inequa_original}
            \left\|C_+(\overline{f}g)\right\|_{L^2}\leq \frac{1}{\sqrt{2\pi}}\left\|f\right\|_{L^2}\left\||\partial|^{\frac{1}{2}}g\right\|_{L^2}.            
        \end{equation}    
\end{lemma}
\begin{proof}
    Applying the Fourier transform, it follows that
    \begin{align*}
        \widehat{C_+(\bar{f}g)}(\xi)&=\mathbf{1}_{\xi\geq 0}\widehat{\bar{f}g}(\xi)\\
        &=\frac{1}{\sqrt{2\pi}}\int^\infty_0 \widehat{\bar{f}}(\xi-\eta)\widehat{g}(\eta)\ d\eta\\
        &=\frac{1}{\sqrt{2\pi}}\int^\infty_0 \overline{\widehat{f}(\eta-\xi)}\widehat{g}(\eta)\ d\eta.
    \end{align*}
    The change of variable in convolution gives
    $$\widehat{C_+(\overline{f}g)}(\xi)=\frac{1}{\sqrt{2\pi}}\int^\infty_0\widehat{g}(\xi+\eta)\overline{\widehat{f}(\eta)}\ d\eta\quad \text{for}\quad \xi\geq 0.$$
    By Cauchy-Schwarz we have
    $$ \left|\widehat{C_+(\overline{f}g)}(\xi)\right|^2 \leq \frac{1}{2\pi} \left( \int_0^\infty \left|\widehat{f}(\eta)\right|^2\,d\eta \right) \left( \int_0^\infty \left|\widehat{g}(\xi+\eta)\right|^2\,d\eta\right),$$
    which implies that
    $$\int^\infty_0\left|\widehat{C_+(\overline{f}g)}(\xi)\right|^2 d \xi\leq \frac{1}{2\pi}\left(\int^\infty_0\int^\infty_0 \left|\widehat{g}(\xi+\eta)\right|^2d\xi d\eta\right)\left(\int^\infty_0\left|\widehat{f}(\eta)\right|^2d\eta\right).$$
    Set, $\zeta\coloneqq \xi+\eta$, for $\eta\in [0,\infty)$, by Tonelli's theorem it follows that
    $$\int^\infty_0\zeta\left|\widehat{g}(\zeta)\right|^2d\zeta=\int^\infty_0\int^\infty_0 \left|\widehat{g}(\xi+\eta)\right|^2d\xi d\eta.$$
   Hence, by Plancherel's theorem,
    $$
    \|C_+(\overline{f}g)\|_{L^2}^2
    \leq
    \frac{1}{2\pi}
    \|f\|_{L^2}^2
    \bigl\||\partial|^{\frac12}g\bigr\|_{L^2}^2.
    $$
    Thus, taking square roots yields \eqref{GL_type_inequa_original}.
\end{proof}
We next present a lemma of an integration-by-parts identity in Fourier space, which will be used repeatedly. This allows us to redistribute derivatives in the nonlinear terms without losing regularity. 
\begin{lemma}\label{Commutator_type}
Let $f,g \in H^{\frac{3}{2}}_+(\mathbb{R})$. Then the following identity holds:
\begin{equation}\label{lemma_cheatcode}
\int \overline{f}\,g\,C_+(g\,\overline{f}')\dx
= - \int \overline{f}\,g'\,C_+(g\,\overline{f})\dx
- \int \overline{f}\,g\,C_+(g'\,\overline{f})\dx.
\end{equation}
\end{lemma}

\begin{proof}
Note that \eqref{lemma_cheatcode} is not merely an application of integration by parts, since we need to avoid the derivative hitting $f$ in the later proof. However, although the traditional integration by parts is not working, we still can employ the trick in Fourier space. Here, we have
    \begin{align*}
        \int \overline{f}\,g\,C_+(g\,\overline{f})\dx&=\langle f\overline{g},C_+(g\overline{f})\rangle\\
        &=\langle \widehat{f\overline{g}},\widehat{C_+(g\overline{f})}\rangle\\
        &= \int_{\xi\geq 0}\overline{\widehat{f\overline{g}}(\xi)}\cdot \widehat{g\overline{f}}(\xi)\,\dd\xi.
    \end{align*}
    Notice that for $\overline{\widehat{f\overline{g}}(\xi)}$, we have
    $$\overline{\widehat{f\overline{g}}(\xi)}=\frac{1}{\sqrt{2\pi}}\int \overline{\widehat{\overline{g}}(\xi-\sigma)\widehat{f}(\sigma)}\,\dd\sigma=\frac{1}{\sqrt{2\pi}}\int \widehat{g}(\sigma-\xi)\overline{\widehat{f}(\sigma)}\,\dd\sigma$$
    and for $\widehat{g\overline{f}}(\xi)$ we have,
    $$\widehat{g\overline{f}}(\xi)=\frac{1}{\sqrt{2\pi}}\int \widehat{g}(\xi-\eta)\widehat{\overline{f}}(\eta)\,\dd \eta=\frac{1}{\sqrt{2\pi}}\int \widehat{g}(\xi-\eta)\overline{\widehat{f}(-\eta)}\,\dd \eta.$$
    Therefore, we have 
    \begin{equation}\label{1/2IBP_cheatcode}
         \int \overline{f}\,g\,C_+(g\,\overline{f})\dx
         =\frac{1}{2\pi}\iiint \overline{\widehat{f}(\sigma)}\widehat{g}(\sigma-\xi)\widehat{g}(\xi-\eta)\overline{\widehat{f}(-\eta)}\,\mathbf{1}_{\xi\geq 0}\,\dd\xi\,\dd\eta\,\dd\sigma.
    \end{equation}

    On the other hand, integration by parts yields
    \begin{align*}
        \int \overline{f}gC_+(g\overline{f}')\dx&=\frac{1}{2}\int \overline{f}gC_+(g\overline{f}')- \overline{f}'gC_+(g\overline{f})-\overline{f}g'C_+(g\overline{f})-\overline{f}gC_+(g'\overline{f})\dx.
    \end{align*}
    Comparing this with \eqref{lemma_cheatcode}, we see that we must remove the first two terms on the right-hand side above.
    By \eqref{1/2IBP_cheatcode}, we have
    \begin{align*}
         \frac{1}{2}\int \overline{f}&gC_+(g\overline{f}')- \overline{f}'gC_+(g\overline{f})\dx\\&=\frac{1}{4\pi}\iiint \overline{\widehat{f}(\sigma)}\widehat{g}(\sigma-\xi)\widehat{g}(\xi-\eta)\overline{\widehat{f}(-\eta)}\big[i\eta-i\sigma\big]\,\mathbf{1}_{\xi\geq 0}\,\dd\xi\,\dd\eta\,\dd\sigma\\
        &=\frac{1}{4\pi}\iiint \overline{\widehat{f}(\sigma)}\widehat{g}(\sigma-\xi)\widehat{g}(\xi-\eta)\overline{\widehat{f}(-\eta)}\big[-i(\xi-\eta)-i(\sigma-\xi)\big] \cdot \mathbf{1}_{\xi\geq 0}\,\dd\xi\,\dd\eta\,\dd\sigma\\
        &=-\frac{1}{2}\int \overline{f}gC_+(g'\overline{f})+\overline{f}g'C_+(g\overline{f})\dx.
    \end{align*}
    Therefore, 
    $$\int \overline{f}\,g\,C_+(g\,\overline{f}')\dx
= - \int \overline{f}\,g'\,C_+(g\,\overline{f})\dx
-\int \overline{f}\,g\,C_+(g'\,\overline{f})\dx.$$
\end{proof}

\begin{comment}
\section{CCM and CCM$_\varepsilon$}

\begin{equation}\tag{CCM$_\varepsilon$}\label{CCM2}
i\tfrac{d}{dt} q-i\varepsilon q_{xx}= -q'' \pm 2iq C_+\big( |q|^2\big)' .
\end{equation}
The unknown field $q(t,x)$ is a complex-valued function of $x\in\R$.  We further demand that $q(t)$ belongs to the Hardy space
$$
L^2_+(\R) := \{f\in L^2(\R):\, \widehat{f}(\xi) = 0 \text{ for } \xi<0\} .
$$
The operator $C_+$ appearing in the nonlinearity of (CCM$_\varepsilon$) denotes the Cauchy--Szeg\H o projection from $L^2(\R)$ onto $L^2_+(\R)$.%  Our convention for the Fourier transform is given in .
\end{comment}

%%%%%%%%%%%%%%%%%%%%%%%%%%%%%%%%%%%%%%%%%%%%%%%%%%%%%%%%%%%%%%%%%%%%%%%%%%%%%%%%%%%%%%%%%%%%%%%%%%%%%%%%
\section{CCM$_\varepsilon$ local well-posedness}

In this section, we establish the local well-posedness of the regularized equation \eqref{CCM2} in $H^s_+(\R)$ for $s\geq 2$.
\begin{proposition}\label{CCM_eps_lwp}
    Fix an integer $s\geq 2$.  Given $A>0$, there exists $T>0$ so that for any initial data $q(0)\in H^s_+$ satisfying $\| q(0) \|_{H^s} \leq A$, there exists a unique solution $q(t) \in C([0,T];H^s_+)$ to \eqref{CCM2}.  Moreover, the map $(t,q(0))\mapsto q(t)$ is jointly continuous from $[0,T]\times H^s_+$ into $H^s_+$. 
\end{proposition}
\begin{proof}
    We will first write a Duhamel's formula. Here, we have 
    $$q_t=(\varepsilon+i)q''\pm 2qC_+(\vert q\vert^2)'.$$
    By applying Duhamel's principle it gives
    $$q(t)=\mathcal{S}_\varepsilon(t)q(0)\pm  2\int^t_0 \mathcal{S}_\varepsilon (t-t')\cdot[qC_+(\vert q\vert^2)'](t')\,dt',$$
    where we define the linear flow operator
    
    $$\mathcal{S}_\varepsilon(t)f=\left[e^{-(\varepsilon+i)\xi^2t}\hat{f}(\xi)\right]^\vee.$$
    Let $q\in H^s_+$ and define a mapping $\Phi:([0,T];H^s_+(\R))\to([0,T];H^s_+(\R))$, we have
    $$\Phi(q)(t)=\mathcal{S}_\varepsilon(t)q(0)\pm 2\int^t_0 \mathcal{S}_\varepsilon (t-t')\cdot[qC_+(\vert q\vert^2)'](t')\,dt'\, .$$
    We first show that the linear flow operator $\mathcal{S}_\varepsilon$ is mapping $H^{s-1}_+$ to $H^s_+.$ Let, $m_\eps:=e^{-(\eps+i)\xi^2 t}$ be the Fourier multipliers of the linear flow operator $S_\eps$. We have the modulus of the Fourier multipliers, $\vert m_\eps(\xi,t)\vert =e^{-\varepsilon\xi^2t}\leq 1$ and we have
    \begin{align*}
    \|\mathcal S_\varepsilon(t)f\|_{H^{s}_+}^{2}
    &=\int_{0}^{\infty}(1+\xi^{2})^{s}\;
       e^{-2\varepsilon\xi^{2}t}\,
       |\hat f(\xi)|^{2}\,d\xi\\
    &\le\int_{0}^{\infty}(1+\xi^{2})e^{-2\varepsilon\xi^2 t}\,|\hat f(\xi)|^{2}(1+\xi^2)^{s-1}\,\dd\xi\\
    &\lesssim C_{\varepsilon,t}\left\|f\right\|_{H^{s-1}_+}^{2},
    \end{align*}
    where we have $C_{\varepsilon,t}=\max \left\{1, \frac{e^{2\varepsilon t-1}}{2\varepsilon t}\right\}.$
     For the nonlinearity, by noticing that $H^{s-1}$ is an algebra when $s\geq 2$, it follows that
      \begin{align*}
        \left\|\int^t_0\mathcal{S}_{\varepsilon}(t-t')\left[qC_+(\vert q\vert^2)'\right](t')\,\dd t'\right\|_{H^s_+}&\leq \int^t_0 \snorm{S_\eps (t-t')\left[qC_+(|q|^2)'(t')\right]}dt'\\
         &\lesssim \int^t_0 \sqrt{C_{\eps,t-t'}} \left\|qC_+\bigl(|q|^2\bigr)'\right\|_{H^{s-1}_+}(t')\;dt'\\
         &\lesssim \int^t_0 \sqrt{C_{\eps,t-t'}} \|q(t')\|_{H^s_{+}}^3dt'\\
         &\leq C\left(\int_0^t \sqrt{C_{\eps,t-t'}}    dt'\right)\cdot \|q(t')\|_{H^s_{+}}^3\\
         &=: \Theta_\eps (t)\cdot \|q\|_{C([0,T];H^s_{+})}^3
    \end{align*}
    We next show that the constant $\Theta_\eps(t)$ is finite by employing a change of variables. Say, $\tau=t-t'$ and therefore we have 
    $$\Theta_\eps(t)=C\int^t_0 \sqrt{C_{\eps,\tau}}\,d\tau.$$
    By the previous computation of $C_{\eps,t}$, we have $$\sqrt{C_{\eps,\tau}}=\max \left\{1,\frac{e^{\eps\tau-\frac{1}{2}}}{\sqrt{2\eps\tau}}\right\}\leq 1+\frac{e^{\eps\tau-\frac{1}{2}}}{\sqrt{2\eps\tau}}.$$
    Hence, 
    \begin{align*}
        \Theta_\eps(t)&\lesssim \int^t_0 1+\frac{e^{\eps\tau-\frac{1}{2}}}{\sqrt{2\eps\tau}}d\tau\\
        &= t+\frac{1}{\sqrt{2\eps}}\int^t_0 e^{\eps \tau-\frac12}\tau^{-\frac{1}{2}}d\tau. 
    \end{align*}
    Since $0\leq \tau\leq t$, we have
    $$e^{\eps \tau-\frac12}\leq e^{\eps t-\frac12}.$$
    Thus, 
    $$\begin{aligned}
    \Theta_{\varepsilon}(t)
    &\lesssim t+\frac{e^{\varepsilon t-\frac{1}{2}}}{\sqrt{2\varepsilon}}
    \int_0^t \tau^{-\frac{1}{2}}\,d\tau\\
    &=t+\frac{e^{\varepsilon t-\frac{1}{2}}}{\sqrt{2\varepsilon}}
    \cdot 2\sqrt{t}\\
    &\lesssim t+e^{\varepsilon t-\frac{1}{2}}
    \sqrt{\frac{2t}{\varepsilon}}\\
    &<\infty
    \end{aligned}$$
    Also, the constant $\Theta_\eps(t)$ goes to $0$ as $t\to 0^+$. Hence,
    \begin{align*}
        \| \Phi(q) \|_{C([0,T]; H^s_+)} \leq \| q_0 \|_{H^s_+} + \Theta_\eps(T) \| q \|_{C([0,T]; H^s_+)}^3.
    \end{align*}
    Set $R:=2A$ and since $\Theta_\eps(T)\to 0$ as $T\to 0^+$, we may choose a time $T$ sufficiently small so that 
    $$\Theta_\eps(T) \leq \frac{1}{6R^2}.$$
    Then, for every $q\in B_R$, we have $$\begin{aligned}\|\Phi(q)\|_{C([0,T];H^s_+)}&\leq\|q_0\|_{H^s_+}+\Theta_\varepsilon(T)\|q\|_{C([0,T];H^s_+)}^3\\
    &\leq\frac{R}{2}+\Theta_\varepsilon(T)R^3\\
    &\leq\frac{R}{2}+\frac{R}{2}=R.
    \end{aligned}$$
    We further employ the contracting mapping theorem, for any $q_1,q_2\in B_{2\|q_0\|_{H^s_+}}$ we have:   
    \begin{align*}
        \|\Phi(q_1)&-\Phi(q_2)\|_{C([0,T]; H^s_+)}
        \\ &\leq \Theta_\eps(T)\left(\|q_1\|^2_{C([0,T]; H^s_+)} +\|q_1\|_{C([0,T]; H^s_+)} \|q_2\|_{C([0,T]; H^s_+)} +\|q_2\|^2_{C([0,T]; H^s_+)} \right)\\
        &\phantom{==}\cdot \|q_1-q_2\|_{C([0,T]; H^s_+)} .
    \end{align*}
    Since $q_1,q_2\in B_R$, it follows that
    $$\|q_1\|_{C([0,T]; H^s_+)} ,\|q_2\|_{C([0,T]; H^s_+)} \leq R ,$$
    and so
    \begin{align*}
        \|\Phi(q_1)-\Phi(q_2)\|_{C([0,T]; H^s_+)} &\leq 3\Theta_\eps (T)R^2\|q_1-q_2\|_{C([0,T]; H^s_+)}\\
        &\leq \frac12\|q_1-q_2\|_{C([0,T]; H^s_+)} .
    \end{align*}    
    we see that $\Phi$ will be a strict contraction on $H^s$, applying the contraction mapping theorem, we obtain a fixed point to $\Phi$, which proved the well-posedness of \eqref{CCM2} in $H^s$.
\end{proof}
%%%%%%%%%%%%%%%%%%%%%%%%%%%%%%%%%%%%%%%%%%%%%%%%%%%%%%%%%%%%%%%%%%%%%%%%%%%%%%%%%%%%%%%%%%%%%%%%%%%%%%%%

\section{CCM$_\varepsilon$ a-priori estimates}

In this section, we establish a-priori estimates for solutions to \eqref{CCM2}. For $\eps\in (0,1]$, let 
\begin{equation} 
q_\eps (t) \quad\text{denote the solution to \eqref{CCM2} with initial data}\quad q_\eps (0)=P_{\leq \eps^{-\frac{1}{3}}}q_0.
\label{q0}
\end{equation}
Note that $q_\eps(0)$ is in $H^\infty_+$, and so there exists a local-in-time smooth solution $q_\eps(t)$ by Theorem \ref{CCM_eps_lwp}.
We first introduce the Lax operator
$$
\mathcal{L}_q:=-i\partial_x\mp qC_+\overline{q}.
$$ The following hierarchy of conservation laws for \eqref{CCM}
was established in \cite[\S~1.3]{GérardLenzmann2024}.

\begin{proposition}\label{jul29_1601}
    If $q(t)\in C([0,T];H^{\frac{n}{2}}_+(\R))$ with some $n\in \N$ solves \eqref{CCM}, the quantities 
    $$E_k(q)\coloneqq \langle q,\mathcal{L}^k_q q\rangle\quad \text{with }k=0,\dots,n$$
    are conserved. That is,
    $$\frac{\dd}{dt}E_n=0\quad \forall n\geq 0.$$
\end{proposition}

We begin with the lowest-order \emph{a-priori} bound of \eqref{CCM2}.
\begin{proposition}
    For any $q_0\in L^2_+$, the solution \eqref{q0} satisfies
    \begin{equation}\label{apr14_1446}
        \snorm{q_\varepsilon(t)}^2_{L^2}\leq \snorm{q_\eps(0)}^2_{L^2}
    \end{equation}
    for all $t\geq 0$ and $\varepsilon>0$.
\end{proposition}
\begin{proof}
We aim to establish an \textit{a-priori} bound for $E_0$, which is defined in Proposition~\ref{jul29_1601}. In the following, we let $q = q_\eps$ denote the solution \eqref{q0}. We have 
\begin{align*}
    \frac{d}{dt}\int \vert q\vert^2\;dx&=\int \overline{q}q_t+q \overline{q}_t\;dx\\
    &=\int \overline{q}\left(\varepsilon q''+iq''\pm 2qC_+(\vert q\vert^2)'\right)+q\left(\varepsilon \overline{q}''-i\overline{q}''\mp 2\overline{q}C_-(\vert q\vert^2)'\right)\;dx\\
    &=\varepsilon\int \overline{q} q''+q\overline{q}''\;dx\\
    &=-\varepsilon\int\overline{q}'q'+q'\overline{q}'\;dx\\
    &=-2\varepsilon \int \vert q'\vert^2\;dx\leq 0.
\end{align*}
Hence,
$$
\frac{d}{dt}E_0(t)
  = -2\varepsilon\int|q'|^{2}
  \;\le\;0,
$$
which implies
$$
    E_0(t)\le E_0(0),\;\forall t\ge0.
$$
\end{proof}

\begin{proposition}\label{jul31_1405}

Let $q_0\in H^2_+(\mathbb{R})$, and for $\varepsilon\in(0,1]$, let $q_\varepsilon$ be the solution to $(\mathrm{CCM}_\varepsilon)$ defined in \eqref{q0}. In the focusing case, assume in addition that
$$\|q_0\|_{L^2}^2<2\pi.$$
Then there exists a constant $C=C(q_0)>0$, independent of $\varepsilon\in(0,1]$, such that for every $t>0$,$$\|q_\varepsilon(t)\|_{\dot{H}^{\frac{1}{2}}}^2\leq C+C\varepsilon t\left(1+\|q_\varepsilon\|_{C([0,t];H^2)}^2\right).$$

\end{proposition}
\begin{proof}
    For 
$$E_1=\langle q,\mathcal{L}q\rangle=\int \left[(-i\partial_x \mp q\,C_+\overline{q})q\right]\cdot \overline{q}\,\mathrm{d}x=\int -i\overline{q}q'\mp \frac{1}{2}\vert q\vert^4\,\mathrm{d}x,$$
taking a time derivative yields
\begin{align}
    \frac{d}{dt}E_1&=\int i\overline{q}'q_t-iq'\overline{q}_t\mp \vert q\vert^2(q_t\overline{q}+q\overline{q}_t)\,\dd x\nonumber\\
    &=\varepsilon\int i\overline{q}'q''-iq'\overline{q}''\mp \vert q\vert^2(q''\overline{q}+q\overline{q}'')\,\dd x\nonumber\\
    &=-2\varepsilon\left\| q\right\|^2_{\dot{H}^{\frac{3}{2}}}\mp\,\varepsilon\int \vert q\vert^2(q''\overline{q}+q\overline{q}'')\,\dd x.\label{aug7_1613}
\end{align}
Set $I=\int |q|^2(q''\overline{q}+q\overline{q}'')$, by Hölder's inequality, it follows that 
$$|I|\leq 2\|q''\|_{L^2}\|q\|^3_{L^6}.$$
By Sobolev embedding in one dimension and Hölder's inequality in Fourier space, it follows that
$$\|q\|_{L^6}\lesssim \|q\|_{\dot{H}^{\frac{1}{3}}}\lesssim \|q\|^{\frac16}_{\dot{H}^2}\|q\|^{\frac56}_{L^2}.$$
Therefore, the following estimates are followed by using Gagliardo--Nirenberg inequality and Young's inequality
\begin{align*}
    |I|&\leq 2\|q''\|_{L^2}\|q\|^3_{L^6}\\
    &\lesssim \|q''\|_{L^2}^{\frac32}\|q\|^{\frac{5}{2}}_{L^2}\\
    &\lesssim \|q''\|^2_{L^2}+\|q\|^{10}_{L^2}.
\end{align*}
Since the term $-2\eps\|q\|^2_{\dH^{\frac{3}{2}}}$ is negative, we therefore ignore the contribution which made by this term. Now, we have 
$$\frac{d}{dt}E_1(t)\leq C(M_0)\eps (1+\|q(t)\|_{\dH}^2)$$
and after integration, it follows that
$$E_1(t)\leq E_1(0)+C(M_0)\eps t(1+\|q\|^2_{C([0,t];\dH^2)})$$
In the focusing case, we have
$$E_1(q)=\|q\|^2_{\dH^{\frac12}}-\snorm{C_+(|q|^2)}^2_{L^2}.$$
The second term follows by the estimate that
$$\snorm{C_+(|q|^2)}^2_{L^2}\leq \frac{\|q\|^2_{L^2}}{2\pi}\|q\|^2_{\dH^{\frac12}}.$$
Thus,
$$E_1\left(q(t)\right)\geq (1-\frac{M_0}{2\pi})\|q(t)\|^2_{\dH^{\frac{1}{2}}}$$
which $M_0<2\pi.$
Hence, we have
$$\|q\|^2_{\dH^{\frac12}}\leq C+C\eps t(1+\|q\|^2_{C([0,t];\dH^2)}).$$
\end{proof}

\begin{proposition}\label{may6_0905}
Let $q_\eps$ solve \eqref{CCM2} with $q(0)\in H^{3}(\R)\cap L^2_+(\R)$. 
Then there exists a constant $C>0$, depending only on $\|q(0)\|_{L^2}$ but independent of $\varepsilon\in(0,1]$, such that for all $t>0$,
$$
E_4(t)\le Ct(\snorm{q_\eps}_{C_t\dH^{\frac{1}{2}}}^{10}+1)+E_4(0),
$$
where $E_4$ is the fourth-level energy.
\end{proposition}
\begin{proof}
In the following, we let $q = q_\eps$ denote the solution \eqref{q0}. 
After applying the Proposition \ref{jul29_1601} 
with $k=4$, we have the following terms where writing for short, we define that
$$\frac{d}{dt} E_4 = \sum_{j=1}^7 I_j$$
where
\begin{align*}
    I_1&\coloneqq \langle -i\partial (-i\partial q),-i\partial (-i\partial q)\rangle=\langle -\partial^2 q, -\partial^2 q \rangle,\\
    I_2&\coloneqq \langle-i\partial(iqC_+(\overline{q}q)),-i\partial(iqC_+(\overline{q}q))\rangle,\\
    I_3&\coloneqq  \langle (-i\partial q)C_+(\overline{q}q), (-i\partial q)C_+(\overline{q}q) \rangle,\\
    I_4&\coloneqq \langle qC_+(\overline{q}q)', qC_+(\overline{q}q)' \rangle,\\
    I_5&\coloneqq  \langle qC_+( \overline{q} q C_+ (\overline{q}q)), qC_+( \overline{q} q C_+ (\overline{q}q))\rangle,\\
    I_6&\coloneqq  \mp 2 \langle -\partial^2 q, qC_+(\overline{q} \partial q) \rangle \mp 2 \langle -\partial^2 q, (-i\partial q)C_+(\overline{q}q) \rangle,\\
    I_7&\coloneqq  \mp 2 \langle -\partial^2 q, qC_+(\overline{q}q)' \rangle
    \mp 2 \langle -\partial^2 q, qC_+(\overline{q} qC_+(\overline{q}q)) \rangle.
\end{align*} 
In the proof, we are bounding the terms through $I_1$ to $I_7$. \begin{comment}Here are the PDE that need to plug in:
    \begin{align*}
    q_t &= \varepsilon q'' + i q'' \pm 2 q\, C_+\big( |q|^2 \big)'\\
    \overline{q}_t &= \varepsilon \overline{q}'' - i \overline{q}'' \mp 2 \overline{q}\, C_-\big( |q|^2 \big)'.
    \end{align*}    
    \bigskip
\noindent
\end{comment}
\\\\
\noindent
{\it \underline{Estimate of $I_1$}}.
\begin{align}
    \frac{d}{dt}I_1 &= 2\Re\int \overline{q}''q''_tdx \nonumber \\
    &=2\Re\int \overline{q}''\left\{ \varepsilon q'' + i q'' \pm 2 q\, C_+\big( |q|^2 \big)'\right\}''dx\nonumber\\
    &=-2\Re\int \varepsilon \vert q'''\vert dx \nonumber \\
    &=-2\varepsilon \left\| q'''\right\|^2_{L^2}. \nonumber
\end{align}
{\it \underline{Estimate of $I_2$}}.
We have $I_2=\langle-i\partial(iqC_+(\overline{q}q)),-i\partial(iqC_+(\overline{q}q))\rangle$.
Therefore,
\begin{align}
    \frac{d}{dt}I_2&=-2\Re\int [(qC_+(\vert q\vert^2))]_t\overline{(qC_+(\vert q\vert^2))}''dx\nonumber\\
    &=-2\Re\int \big(q_tC_+(\overline{q}q)+qC_+(\overline{q}_tq)+qC_+(\overline{q}q_t)\big)\cdot\overline{(qC_+(\vert q\vert^2))}''dx\label{apr21_1248}
\end{align}
where \eqref{apr21_1248} can be reduced to 
\begin{multline}\label{apr21_1249}
    2\Re \int q_tC_+(\overline{q} q) \cdot \overline{(q C_+(|q|^2))}''\, dx-2\Re \int qC_+(\overline{q}_t q) \cdot \overline{(q C_+(|q|^2))}''\, dx\\ -2\Re \int qC_+(\overline{q} q_t) \cdot \overline{(q C_+(|q|^2))}''\, dx.
\end{multline}
Plugging \eqref{CCM2} into \eqref{apr21_1249} yields
\begin{align*}
\varepsilon \int S_\alpha \, T_\beta \, dx
+\varepsilon \int S_\alpha \, T_\gamma \, dx
\eqqcolon \mathcal{A}+\mathcal{B}.
\end{align*}
where
\begin{align*}
    S_\alpha&:=\sum_{(\alpha_1,\alpha_2,\alpha_3)\in \{(2,0,0),(0,2,0),(0,0,2)\}} q^{(\alpha_1)}C_+(\overline{q}^{(\alpha_2)}q^{(\alpha_3)}),\\
    T_\beta&:=\sum_{(\beta_1,\beta_2,\beta_3)\in \{(2,0,0),(0,2,0),(0,0,2)\}}c_\beta\left( q^{(\beta_1)}C_+(\overline{q}^{(\beta_2)}q^{(\beta_3)})\right),\\
    T_\gamma&:=\sum_{\substack{0\leq \gamma_1,\gamma_2,\gamma_3<2\\ \gamma_1+\gamma_2+\gamma_3=2\\ \gamma\in \mathbb{Z}}} c_\gamma\left( q^{(\gamma_1)}C_+(\overline{q}^{(\gamma_2)}q^{(\gamma_3)})\right).
\end{align*}
Both $\mathcal{A}$ and $\mathcal{B}$ are sextic in $q$. Thus, by Hölder's inequality, we place each factor in $L^6$, so that the Cauchy--Szeg\H{o} operator is bounded on $L^p$.
Considering $\vert \mathcal{A}\vert$, by Sobolev embedding and the Gagliardo--Nirenberg inequality we have the estimate:
\begin{align}
    \vert \mathcal{A}\vert&\lesssim \varepsilon \left\| q''\right\|^2_{L^{6}}\left\| q\right\|^4_{L^{6}}\nonumber \\
    &\lesssim \varepsilon\left\| q\right\|^{2}_{\dot{H}^{\frac{7}{3}}}\left\| q\right\|^{4}_{\dot{H}^{\frac{1}{3}}}\nonumber \\
    &\lesssim_{s} \varepsilon \left\| q\right\|^{\frac{8}{15}}_{\dot{H}^{\frac{1}{2}}}\cdot \left\| q\right\|^{\frac{22}{15}}_{\dot{H}^3}\nonumber \\
    &\leq \frac{1}{100}\left\| q\right\|^2_{\dH^3}+C\left\| q\right\|^2_{\dH^{\frac{1}{2}}}.\label{L6_case_2}
\end{align}
Considering $\vert \mathcal{B}\vert$, by Sobolev embedding and Gagliardo--Nirenberg inequality we have the following estimate:
\begin{align}
    \vert \mathcal{B}\vert&\lesssim \varepsilon\left\| q''\right\|_{L^6}\left\| q'\right\|^2_{L^6}\left\| q\right\|^3_{L^6}\nonumber\\
    &\lesssim_s \varepsilon \left(\left\| q\right\|^{\frac{4}{15}}_{\dot{H}^{\frac{1}{2}}}\cdot \left\| q\right\|^{\frac{11}{15}}_{\dot{H}^3}\right)\left(\left\| q\right\|^{\frac{4}{3}}_{\dot{H}^{\frac{1}{2}}}\cdot \left\| q\right\|^{\frac{2}{3}}_{\dot{H}^3}\right)\nonumber \\
    &\leq \frac{1}{100}\left\| q\right\|^2_{\dH^3}+C\left\| q\right\|^{\frac{16}{3}}_{\dH^{\frac{1}{2}}}.
    \label{L6_case_1}
\end{align}
Thus, it gives
\begin{equation}\label{I2_final}
    \vert \eqref{apr21_1249}\vert \leq \frac{c_\alpha+c_\gamma}{100}\left\| q\right\|^2_{\dH^3}+C\left\| q\right\|^{\frac{16}{3}}_{\dH^{\frac{1}{2}}}.
\end{equation}
\\\\
\noindent
{\it\underline{Estimate of $I_3$}}.
We have $I_3\coloneqq  \langle (-i\partial q)C_+(\overline{q}q), (-i\partial q)C_+(\overline{q}q) \rangle$. Therefore,
\begin{align}
    \frac{d}{dt}I_3&=-2\Re\int [q'C_+(\overline{q}q)]_{t}\cdot\overline{q'C_+(\overline{q}q)}\,dx\nonumber\\
    &=-2\Re\int q'_tC_+(\overline{q}q)\cdot\overline{q'C_+(\overline{q}q)}dx-2\Re\int q'C_+(\overline{q}_tq)\cdot\overline{q'C_+(\overline{q}q)}dx\nonumber\\
    &\phantom{==}-2\Re\int q'C_+(\overline{q}q_t)\cdot\overline{q'C_+(\overline{q}q)}\,dx\nonumber \\
    &\eqqcolon \mathcal{A}+\mathcal{B}+\mathcal{C}\label{apr21_1301}.
\end{align}
where can be further reduced to 
\begin{multline}\label{apr21_1308}
    -2\Re\int q'_tC_+(\overline{q}q)\cdot\overline{q'C_+(\overline{q}q)}dx-2\Re\int q'C_+(\overline{q}_tq)\cdot\overline{q'C_+(\overline{q}q)}dx\\-2\Re\int q'C_+(\overline{q}q_t)\cdot\overline{q'C_+(\overline{q}q)}\,dx 
    \eqqcolon \mathcal{A}+\mathcal{B}+\mathcal{C}.
\end{multline}
For $\vert \mathcal{A}\vert$, after choosing a set of Hölder pair we have the estimate
\begin{align*}
    \vert \mathcal{A}\vert&=2\varepsilon \Re \int q'''C_+(\overline{q}q)\cdot\overline{q'C_+(\overline{q}q)}dx\\
    &\lesssim \varepsilon \left\| q'''\right\|_{L^2}\left\| q'\right\|_{L^{10}}\left\| q\right\|^4_{L^{10}}\\
    &\lesssim_s \varepsilon \left\| q\right\|^{\frac{16}{25}}_{\dot{H}^{\frac{1}{2}}}\left\| q\right\|^{\frac{34}{25}}_{\dot{H}^{3}}\\
    &\leq \frac{1}{100}\left\| q\right\|^2_{\dH^3}+C\left\| q\right\|^2_{\dH^{\frac{1}{2}}}.
\end{align*}
We next show the estimate for the term $\mathcal{B}$, and notice that the term $\mathcal{C}$ would follow exactly the same estimate of $\mathcal{B}$. For $\mathcal{B}$, it follows that
\begin{align*}
    \vert \mathcal{B}\vert &=2\varepsilon \Re \int q'C_+(\overline{q}''q)\cdot\overline{q'C_+(\overline{q}q)}dx\\
    &\lesssim \varepsilon \left\| q''\right\|_{L^{6}}\left\| q'\right\|^2_{L^6}\left\| q\right\|^3_{L^6}\\
    &\lesssim_s \varepsilon \left(\left\| q\right\|^{\frac{4}{15}}_{\dot{H}^{\frac{1}{2}}}\cdot \left\| q\right\|^{\frac{11}{15}}_{\dot{H}^3}\right)\cdot \left(\left\| q\right\|^{\frac{4}{3}}_{\dot{H}^{\frac{1}{2}}}\cdot \left\| q\right\|^{\frac{2}{3}}_{\dot{H}^3}\right)\\
    &\leq \frac{1}{100}\left\| q\right\|^2_{\dH^3}+C\left\| q\right\|^{\frac{16}{3}}_{\dH^{\frac{1}{2}}},
\end{align*}
where the last inequality relies on \eqref{L6_case_1}. Thus, it gives
\begin{equation}\label{I3_final}
    \vert \eqref{apr21_1308}\vert \leq \frac{3}{100}\left\| q\right\|^2_{\dH^3}+C\left\| q\right\|_{\dH^{\frac{1}{2}}}^{\frac{16}{3}}.
\end{equation}
\\\\
\noindent
{\it\underline{Estimate of $I_4$}}.
Distributed the derivatives yields
\begin{align*}
I_4=&\langle qC_+(\overline{q}'q), qC_+(\overline{q}'q) \rangle+\langle qC_+(\overline{q}q'), qC_+(\overline{q}q') \rangle\\
&+\langle qC_+(\overline{q}'q), qC_+(\overline{q}q') \rangle+\langle qC_+(\overline{q}q'), qC_+(\overline{q}'q) \rangle
\end{align*}
and we will argue only the first term of $I_4$, whereas the rest of the terms follow exactly the same estimate. Say,
\begin{align}
    \frac{d}{d t}I_{4}&= \frac{d}{d t}I_{4}^{\mathbf{1}}+\left\{\text{similar terms}\right\}\nonumber \\
    &=2\Re\int [qC_+(\overline{q}'q)]_t\cdot \overline{qC_+(\overline{q}'q)}\,d x +\left\{\text{similar terms}\right\} \label{apr21_1317}.
\end{align}
Choosing $L^p$ where $p=6$ yields
\begin{align*}
    \left\vert2\Re\int [qC_+(\overline{q}'q)]_t\cdot \overline{qC_+(\overline{q}'q)}\,d x\right\vert&\leq \varepsilon\left(2\left\| q''\right\|_{L^6}\left\| q'\right\|^2_{L^6}\left\| q\right\|^3_{L^6}\right)+\varepsilon(\left\| q'''\right\|_{L^6}\left\| q'\right\|_{L^6}\left\| q\right\|_{L^6})\\
    &\eqqcolon \mathcal{A}+\mathcal{B}.
\end{align*}
Here, the estimate for $\mathcal{A}$ follows by \eqref{L6_case_1} and $\mathcal{B}$ follows by \eqref{L6_case_2}.
Thus, it gives
\begin{equation}\label{I4_final}
    \vert\eqref{apr21_1317}\vert \leq \frac{8}{100}\left\| q\right\|^2_{\dH^3}+C\left\| q\right\|_{\dH^{\frac{1}{2}}}^{\frac{16}{3}}.
\end{equation}
\\\\
\noindent
{\it\underline{Estimate of $I_5$}}. Recall that we have
$$I_5\coloneqq  \langle qC_+( \overline{q} q C_+ (\overline{q}q)), qC_+( \overline{q} q C_+ (\overline{q}q))\rangle.$$
Taking time derivative yields
\begin{align*}
    \frac{d}{d t} I_5&=2\Re\int \left[ qC_+( \overline{q} q C_+ (\overline{q}q))\right]_t\cdot \overline{ qC_+( \overline{q} q C_+ (\overline{q}q))}\,\dd x .
\end{align*}
Distributing the derivative in order from left to right, we define:
$$
\frac{d}{d t} I_5 := \mathcal{A}+\mathcal{B}+\mathcal{C}+\mathcal{D}+\mathcal{E},
$$
where each term corresponds to the time derivative hitting the respective factor in the expression
$$
qC_+(\overline{q} q C_+(\overline{q}q)).
$$
First, for $\mathcal{A}$, by choosing the right $L^p$ space, it follows that
\begin{align*}
    \vert \mathcal{A}\vert&=2 \varepsilon \left|\int q''C_+( \overline{q} q C_+ (\overline{q}q))\cdot \overline{ qC_+( \overline{q} q C_+ (\overline{q}q))}\,\dd x\right|\\
    &\lesssim \varepsilon \|q''\overline{q} q C_+ (\overline{q}q)\|_{L^2}\| q^2 \overline{q}C_+ (\overline{q}q))\|_{L^2}\\
    &\lesssim \varepsilon \left\| q''\right\|_{L^2}\left\| q\right\|^9_{L^{18}}\\
    &\lesssim \varepsilon \left\| q\right\|_{\dot{H}^2}\left\| q\right\|^9_{\dot{H}^{\frac{4}{9}}}\\
    &\lesssim_s \varepsilon \left\| q\right\|^{\frac{2}{5}}_{\dot{H}^{\frac{1}{2}}}\left\| q\right\|^{\frac{3}{5}}_{\dot{H}^3}\\
    &\leq \frac{1}{100}\left\| q\right\|^2_{\dH^3}+C\left\| q\right\|^{\frac{4}{7}}_{\dH^{\frac{1}{2}}}.
\end{align*}
The remaining terms $\mathcal{B}$ through $\mathcal{E}$ are estimated in the same manner as $\mathcal{A}$, as they share the same structural form.
\\\\
\noindent
{\it\underline{Estimate of $I_6$}}.
Recall that, $$I_6=  \mp 2 \langle -\partial^2 q, qC_+(\overline{q} \partial q) \rangle \mp 2 \langle -\partial^2 q, (-i\partial q)C_+(\overline{q}q) \rangle\eqqcolon I_{6}^{\mathbf{1}}+I_{6}^{\mathbf{2}}.$$
Term $I_{6}^{\mathbf{1}}$ follows the estimate:
\begin{multline*}
    \frac{d}{dt}I_{6}^{\mathbf{1}}=\mp4\Re\int -q''_t\cdot  \overline{qC_+(\overline{q}q')}\,d x\mp4\Re\int -q''\cdot \overline{q_tC_+(\overline{q}q')}\,d x\\ \mp4\Re\int -q''\cdot \overline{qC_+(\overline{q}_tq')}\,d x\mp4\Re\int -q''\cdot \overline{qC_+(\overline{q}q'_t)}\,d x \eqqcolon \mathcal{A}+\mathcal{B}+\mathcal{C}+\mathcal{D}.
\end{multline*}
We now turn to the estimate for $\mathcal{A}$, applying integration by parts gives
\begin{align*}
    \mathcal{A}&=\pm4\Re \int q'_t\cdot \overline{q'C_+(\overline{q}q')}\,d x\pm4\Re \int q'_t\cdot \overline{qC_+(\overline{q}'q')}\,d x\pm4\Re \int q'_t\cdot \overline{qC_+(\overline{q}q'')}\,d x\\
    &\eqqcolon \mathcal{A}_1+\mathcal{A}_2+\mathcal{A}_3.
\end{align*}
Since the the term $q'''$ appears after plugging-in the \eqref{CCM2} and must be placed in $L^2$, the remaining three copies must be assigned to $L^p$ spaces so that Hölder’s inequality is applicable. As there are only three such copies, we place each in $L^6$ to ensure that the exponents sum to 1.
\begin{itemize}
    \item To estimate $\vert \mathcal{A}_1 \vert$, we apply Sobolev embedding and the Gagliardo--Nirenberg inequality,
    \begin{align*}
        \vert\mathcal{A}_1\vert&\lesssim \varepsilon \left\| q'''\right\|_{L^2}\left\| q'\right\|^2_{L^{p}}\left\| q\right\|_{L^{p}}\\
        &\lesssim \varepsilon \left\| q'''\right\|_{L^2}\left\| q'\right\|^2_{L^6}\cdot C_s\cdot C_S\\
        &\lesssim_s \varepsilon \left\| q\right\|_{\dot{H}^3}\cdot \left(\left\| q\right\|^{\frac{4}{3}}_{\dot{H}^{\frac{1}{2}}}\cdot \left\| q\right\|^{\frac{2}{3}}_{\dot{H}^3}\right)\\
        &\leq \frac{1}{100}\left\| q\right\|^2_{\dH^3}+C\left\| q\right\|^8_{\dH^{\frac{1}{2}}}.
    \end{align*}
    \item The term $\vert \mathcal{A}_2 \vert$ is estimated in the same manner as $\vert \mathcal{A}_1 \vert$.
    \item Follow the same routine, the estimate of $\vert \mathcal{A}_3\vert$ would be,
    \begin{align*}
        \vert \mathcal{A}_3\vert&\lesssim \varepsilon\left\| q'''\right\|_{L^2}\left\| q''\right\|_{L^6}\left\| q\right\|^{2}_{L^6}\\
        &\lesssim_s \varepsilon\left\| q\right\|_{\dot{H}^3}\left(\left\| q\right\|^{\frac{4}{15}}_{\dot{H}^{\frac{1}{2}}}\cdot \left\| q\right\|^{\frac{11}{15}}_{\dot{H}^3}\right)\\
        &\leq \frac{1}{100}\left\| q\right\|^2_{\dH^3}+C\left\| q\right\|^2_{\dH^{\frac{1}{2}}}.
    \end{align*}
\end{itemize}
Next we turn to the estimate for $\vert \mathcal{B}\vert$, and we put everything in $L^p$-space where $p=4$.
\begin{align*}
    \vert \mathcal{B}\vert&\lesssim \varepsilon \left\| q''\right\|^2_{L^4}\left\| q'\right\|_{L^4}\left\| q\right\|_{L^4}\\
    &\lesssim_s \varepsilon \left\| q\right\|^2_{\dot{H}^{\frac{9}{4}}}\left\| q\right\|_{\dot{H}^{\frac{5}{4}}}\\
    &\lesssim_s \left(\left\| q\right\|^{\frac{3}{5}}_{\dot{H}^{\frac{1}{2}}}\left\| q\right\|^{\frac{7}{5}}_{\dot{H}^{3}}\right)\cdot \left(\left\| q\right\|^{\frac{7}{10}}_{\dot{H}^{\frac{1}{2}}}\left\| q\right\|^{\frac{3}{10}}_{\dot{H}^{3}}\right)\\
    &\leq \frac{1}{100}\left\| q\right\|^2_{\dH^3}+C\left\| q\right\|^{\frac{26}{3}}_{\dH^{\frac{1}{2}}}\\
    &\leq \frac{1}{100}\snorm{q}^2_{\dH^3}+C(\snorm{q}^{10}_{\dH^{\frac{1}{2}}}+1).
\end{align*}
The estimates for $\mathcal{C}$ and $\mathcal{D}$ are identical to those of $\mathcal{B}$ and $\mathcal{A}_3$, respectively.
We now turn to the estimate of $I^{\mathbf{2}}_6$, it follows that
\begin{multline*}
    \frac{\dd}{\dd t}I^{\mathbf{2}}_6=\mp 4\Re \int -q''_{t}\cdot \overline{q'C_+(\overline{q}q)}\,\dd x\mp 4\Re \int -q''\cdot \overline{q'_tC_+(\overline{q}q)}\,\dd x\mp 4\Re \int -q''\cdot \overline{q'C_+(\overline{q}_tq)}\,\dd x\\
    \mp 4\Re \int -q''\cdot \overline{q'C_+(\overline{q}q_t)}\,\dd x\eqqcolon \mathcal{F}+\mathcal{G}+\mathcal{H}+\mathcal{J}.
\end{multline*}
Consider the term $\mathcal{E}$. Applying integration by parts once, we obtain
\begin{equation*}
    \mathcal{F}=\pm 4\Re\int q'_t\cdot \overline{q''C_+(\overline{q}q)}\,\dd x+\pm 4\Re\int q'_t\cdot \overline{q'C_+(\overline{q}'q)}\,\dd x+\pm 4\Re\int q'_t\cdot \overline{q'C_+(\overline{q}q')}\,\dd x\eqqcolon \mathcal{F}_1+\mathcal{F}_2+\mathcal{F}_3.
\end{equation*}
The estimate for $\mathcal{F}_1$ is identical to that of $\mathcal{A}_3$, and the estimates for $\mathcal{F}_2$ and $\mathcal{F}_3$ follow as in $\mathcal{A}_1$. Also, the estimate for $\mathcal{G}$ follows as $\mathcal{A}_3$, and the estimate for $\mathcal{H}$, $\mathcal{J}$ follow as the estimate for 
$\mathcal{B}$.
\\\\
\noindent
{\it\underline{Estimate of $I_7$}}.
Recall that we have $$    I_7\coloneqq  \mp 2 \langle -\partial^2 q, qC_+(\overline{q}q)' \rangle
    \mp 2 \langle -\partial^2 q, qC_+(\overline{q} qC_+(\overline{q}q)) \rangle\eqqcolon I^{\mathbf{1}}_{7}+I^{\mathbf{2}}_{7}.$$
Here,
$$I^{\mathbf{1}}_{7}=\mp 2 \langle -\partial^2 q, qC_+(\overline{q}'q) \rangle\mp 2 \langle -\partial^2 q, qC_+(\overline{q}q') \rangle$$
so the estimate for the right-hand side of above follows as the estimate for $I^{\mathbf{1}}_{6}$.
We next turn to the estimate for $I^{\mathbf{2}}_7$, it follows that
\begin{align*}
    \frac{\dd}{\dd t}I^{\mathbf{2}}_7&=\mp 4\Re \int -q''_t\cdot \overline{qC_+(\overline{q}qC_+(\overline{q}q))}\,\dd x\,\mp\, 4\Re \int -q''\cdot \overline{q_tC_+(\overline{q}qC_+(\overline{q}q))}\,\dd x\\
    &\mp\, 4\Re \int -q''\cdot \overline{qC_+(\overline{q}_tqC_+(\overline{q}q))}\,\dd x
    \mp\, 4\Re \int -q''\cdot \overline{qC_+(\overline{q}q_tC_+(\overline{q}q))}\,\dd x\\
    &\mp\, 4\Re \int -q''\cdot \overline{qC_+(\overline{q}qC_+(\overline{q}_tq))}\,\dd x\mp\, 4\Re \int -q''\cdot \overline{qC_+(\overline{q}qC_+(\overline{q}q_t))}\,\dd x\\
    &\eqqcolon \mathcal{A}+\mathcal{B}+\mathcal{C}+\mathcal{D}+\mathcal{E}+\mathcal{F}.
\end{align*}
Apply one time of integration by part on $\mathcal{A}$ we have
\begin{equation*}
    \mathcal{A}=\pm 4\Re\int q'_t\cdot \left[\overline{qC_+(\overline{q}qC_+(\overline{q}q))}\right]'\,\dd x\eqqcolon \mathcal{A}\text{-type}.
\end{equation*}
We choose one scenario among the $\mathcal{A}\text{-type}$ and the rest of $4$ scenarios follow as this one. Consider $$  \mathcal{A}=\pm 4\Re\int q'_t\cdot \overline{qC_+(\overline{q}qC_+(\overline{q}q'))}\,\dd x$$ and we choose the $L^p$-space with $p=10$ to satisfy the Hölder pair. It follows that
\begin{align*}
    \vert \mathcal{A}\vert&\lesssim \varepsilon \left\| q'''\right\|_{L^2}\left\| q'\right\|_{L^{10}}\left\| q'\right\|^4_{L^{10}}\\
    &\lesssim_s \varepsilon \left\| q\right\|_{\dot{H}^3}\left\| q'\right\|_{\dot{H}^{\frac{2}{5}}}\\
    &\lesssim_s \varepsilon \left\| q\right\|^{\frac{34}{25}}_{\dot{H}^3}\left\| q\right\|_{\dot{H}^{\frac{1}{2}}}^{\frac{16}{25}}\\
    &\leq \frac{1}{100}\left\| q\right\|^2_{\dH^3}+C\left\| q\right\|^2_{\dH^{\frac{1}{2}}}.
\end{align*}
\\
\noindent
{\it{\underline{Conclusion}}}.
Adding all the terms from $I_2$ to $I_7$, we conclude that
\begin{align*}
    \left|\sum^7_{i=2}\frac{d}{dt}I_i\right|\leq C\left\| q\right\|^2_{\dot{H}^3}+C\left\| q\right\|^{10}_{\dot{H}^{\frac{1}{2}}}.
\end{align*}
Applying Young's inequality, where we fixed each estimate of single terms with the coefficient $\frac{1}{100}\varepsilon$, the terms here are strictly less than 50 so we pick $\frac{50}{100}=\frac{1}{2}$ for an easier computation,
\begin{equation}
    \frac{\dd}{dt}I_1+\vert \sum^7_{j=2}\frac{\dd}{dt}I_j\vert \leq\frac{1}{2}\varepsilon\left\| q\right\|^2_{\dot{H}^2}+C(\snorm{q}_{\dH^{\frac12}}^{10}+1).
\end{equation}
And it gives
\begin{align*}
    -\frac{5}{2}\varepsilon \left\| q\right\|^2_{\dH^2}-C(\snorm{q}_{\dH^{\frac12}}^{10}+1)\leq \frac{\dd}{\dd t}E_4&\leq -2\varepsilon\left\| q\right\|^2_{\dH^2}+\frac{1}{2}\varepsilon \left\| q\right\|^2_{\dH^2}+C(\snorm{q}_{\dH^{\frac12}}^{10}+1)\\
    &\leq C(\snorm{q}_{\dH^{\frac12}}^{10}+1).
\end{align*}
By Grönwall’s inequality, we obtain

\begin{align*}
E_4(t)\le Ct(\snorm{q_\eps}_{C_t\dH^{\frac{1}{2}}}^{10}+1)+E_4(0); \quad \forall t\geq 0, \varepsilon\in (0,1].
\end{align*}
\end{proof}
We next start proving the \emph{a-priori} $H^2$ bound.
\begin{corollary}[$H^2$ bound] \label{may13_1439}
Let $q(t)$ be the solution of \eqref{CCM2} on $[0,T]$. Fix $T>0$, and let $q(t)$ be the solution of \eqref{CCM2} on $[0,T]$.  There exists a constant $C>0$ so that for all $t\in [0,T]$ and $\eps\in (0,1]$ we have
    $$\left\| q\right\|^2_{\dot{H}^2}\leq CE_4(q)+C(\left\| q\right\|^{10}_{\dot{H}^{\frac{1}{2}}}+1).$$
\end{corollary}

\begin{proof}
    For $E_4$, we will prove that $\sum^7_{j=1}I_j$ could be bounded as a combination of $\left\| q\right\|_{\dot{H}^{\frac{1}{2}}}$ and $\left\| q\right\|_{\dot{H}^2}$
    which is defined in the proof of Proposition \ref{may6_0905}.

    For term $I_1$, it gives $I_1=\left\| q\right\|_{\dot{H}^{2}}^2$ directly.

    For term $I_2$, it follows that
        \begin{align*}
            \vert I_2\vert&= 2\Re\int \partial(qC_+(\overline{q}q))\cdot \overline{\partial(qC_+(\overline{q}q))}\,\dd x\\
            &\leq 3\left\| q'\right\|^2_{L^6}\left\| q\right\|^4_{L^6}\\
            &\lesssim 3\left\| q\right\|^{\frac{8}{9}}_{\dH^{\frac{1}{2}}}\left\| q\right\|^{\frac{10}{9}}_{\dot{H}^2}\\
            &\leq \frac{3}{100}\left\| q\right\|^2_{\dH^2}+C\left\| q\right\|^2_{\dH^{\frac{1}{2}}}.
        \end{align*}

    For term $I_3$, it follows by the estimate of term $I_2$ above, which yields
       \begin{align*}
           \vert I_3\vert&=2\Re\int q'C_+(\overline{q}q)\cdot\overline{q'C_+(\overline{q}q)}\,\dd x\\
           &\leq \frac{1}{100}\left\| q\right\|^2_{\dH^2}+C\left\| q\right\|^2_{\dH^{\frac{1}{2}}}.
       \end{align*}    
   
    The term $I_4$ follows by the estimate of term $I_2$ as well where we have $4$ copies.

    Term $I_5$ follows by
       \begin{align*}
           \vert I_5\vert&=2\Re\int  qC_+( \overline{q} q C_+ (\overline{q}q))\cdot \overline{ qC_+( \overline{q} q C_+ (\overline{q}q))}\,\dd x\\
            &\lesssim \left\| q\right\|^{10}_{H^{\frac{2}{5}}}\\
            &\leq \snorm{q}^2_{L^2}\snorm{q}^8_{\dH^{\frac{1}{2}}}.
        \end{align*}

    For the term $I_6$, we define that
            $$I_6= \mp 2 \langle -\partial^2 q, qC_+(\overline{q} \partial q) \rangle \mp 2 \langle -\partial^2 q, (-i\partial q)C_+(\overline{q}q) \rangle\eqqcolon \mathcal{A}+\mathcal{B}.$$
            Term $\mathcal{A}$ gives
            \begin{align}
                \vert \mathcal{A}\vert&=4\Re\int q''\cdot \overline{qC_+(q\overline{q}')}\,\dd x\nonumber\\
                &\lesssim \left\| q\right\|_{\dH^2}\cdot \left\| q\right\|_{\dH^{\frac{4}{3}}}\left\| q\right\|^2_{\dH^{\frac{1}{3}}}\nonumber\\
                &\lesssim_s \left\| q\right\|_{\dH^{\frac{1}{2}}}^{\frac{4}{9}}\left\| q\right\|^{\frac{14}{9}}_{\dH^2}\nonumber\\
                &\leq \frac{1}{100}\left\| q\right\|^2_{\dH^2}+C\left\| q\right\|^2_{\dH^{\frac{1}{2}}}\label{CorE4-I6A-4}.
            \end{align}
            Term $\mathcal{B}$ follows the estimate above.
        
         For term $I_7$, we also define that
            $$I_7=\mp 2 \langle -\partial^2 q, qC_+(\overline{q}q)' \rangle \mp 2 \langle -\partial^2 q, qC_+(\overline{q} qC_+(\overline{q}q)) \rangle\eqqcolon \mathcal{A}+\mathcal{B}.$$
            Term $\mathcal{A}$ gives
            \begin{align*}
                \vert \mathcal{A}\vert &=4\Re \int q''\cdot \overline{qC_+(\overline{q}q)'}\,\dd x.
            \end{align*}
            This follows by the estimate of \eqref{CorE4-I6A-4} but with two copies. For term $\mathcal{B}$, it gives
            \begin{align*}
                \vert \mathcal{B}\vert&=2\Re\int q''\cdot \overline{qC_+(\overline{q}qC_+(\overline{q}q))}\,\dd x\\
                &\lesssim \left\| q\right\|_{\dH^2}\left\| q\right\|^5_{\dH^{\frac{2}{5}}}\\
                &\leq \frac{1}{100}\left\| q\right\|^2_{\dH^2}+\snorm{q}^2_{L^2}\snorm{q}^8_{\dH^{\frac{1}{2}}}.
            \end{align*}
Summing over all such terms, and using a uniform bound of the form
$$
\frac{1}{100}\left\| q\right\|^2_{\dH^2} + C\left\| q\right\|^8_{\dot{H}^{\frac{1}{2}}},
$$
we note that the number of these contributions does not exceed 50. Therefore, for convenience, we collect them into a single term bounded by
$
\frac{1}{2}\left\| q\right\|^2_{H^2} + C\left\| q\right\|^8_{\dot{H}^{\frac{1}{2}}}.
$
It follows that
$$\frac{1}{2}\left\| q\right\|^2_{H^2} - C(\left\| q\right\|^{10}_{\dot{H}^{\frac{1}{2}}}+1)\leq E_4\leq \frac{3}{2}\left\| q\right\|^2_{H^2} + C(\left\| q\right\|^{10}_{\dot{H}^{\frac{1}{2}}}+1),$$
and this gives
$$\left\| q\right\|^2_{\dot{H}^2}\leq \Tilde{C}t+E_4(0)+C(\left\| q\right\|^{10}_{\dH^{\frac{1}{2}}}+1).$$
Here, the power $10$ in the lower-order term is precisely where we use the previous proposition.
We have thus completed the a priori estimate for $E_4(t)$.
\end{proof}

\begin{proposition}\label{july23_0719}
    Fix $q_0$ that defined in \eqref{q0}. In the focusing case, assume in addition that
$$\|q_0\|_{L^2}^2<2\pi.$$
Then, for all $T>0$, there exists $\eps_0>0$ so that the set
$$ \{ q_\eps(t) : t\in[0,T],\ \eps\in(0,\eps_0] \} $$
is bounded in $H^2$.
\end{proposition}
\begin{proof}

By Proposition \ref{jul31_1405}, Proposition \ref{may6_0905}, and Corollary \ref{may13_1439}, we have
\begin{equation}\label{system_inequality}
    \begin{cases}
      \snorm{q(t)}^2_{\dot{H}^{\frac{1}{2}}}&\leq C_1\eps t(\|q\|^2_{\dot{H}^2}+1)+C_1, \\
        \snorm{q(t)}^2_{\dH^2}&\leq C_1t\snorm{q(t)}_{\dH^{\frac{1}{2}}}^{10}+C_1.
    \end{cases}
\end{equation}

Next, we run a standard bootstrap argument.
Fix $T>0$ and let $C$ be larger than both $C_1$ and $\snorm{q(0)}^2_{\dH^{\frac{1}{2}}}$. Let $t_0\in (0,T]$ be the largest time such that 
 \begin{equation}
 \snorm{q(t)}^2_{\dH^{\frac{1}{2}}}\leq 3C \quad\text{for all }t\in [0,t_0].
 \label{boot 1131}
 \end{equation}
 Then \eqref{system_inequality} implies
 $$\|q(t)\|^2_{\dH^2}\leq C(T(3C)^5+1)=:K,$$
 and so
 $$\|q(t)\|_{\dH^{\frac{1}{2}}}^2\leq C\eps T(K+1)+C.$$
 Hence, by choosing $\eps\leq  \frac{1}{T(K + 1)} $ we conclude 
 $$\|q(t)\|^2_{\dH^{\frac{1}{2}}}\leq 2C \quad\text{for all }t\in[0,t_0].$$ 
 Comparing this with \eqref{boot 1131}, we conclude that $t_0 = T$.

\end{proof}

For $H^s$, $s\geq 3$, we proceed by induction.
\begin{proposition}\label{jul29_1403}
    Let $ s \in \mathbb{Z}_+ $ with $ s \geq 3 $ be given. If $ q(0) \in H^s_+ $ and
    $$
    \left\| q(0)\right\|_{L^2}^2 < 2\pi
    $$
    in the focusing case,
    then for any $ T > 0 $, there exists a constants $ \eps_0, C > 0 $ such that
    $$
    \snorm{q_\eps(t)}_{H^s} \leq C
    $$
    for all $t\in [0,T]$ and $\eps\in (0,\eps_0]$.  
\end{proposition}
\begin{proof}
    We will argue using induction. First, assuming inductively that $$\snorm{q(t)}_{H^{n-1}}\leq C$$ for all  $t\in [0,T]$ 
    for some $n\geq 3$. Unlike the previous proofs, we abandon $E_n$ and directly compute the evolution of the $H^n$-norm. It follows that
    \begin{align*}
        \frac{\dd}{dt} \int \vert q^{(n)}\vert^2\,\dd x&=\frac{\dd}{dt} 2\Re\int \ q^{(n)} \overline{q}^{(n)}\,\dd x\\
        &=\varepsilon \int q^{(n+2)}\overline{q}^{(n)}\,\dd x+\varepsilon\int q^{(n)}\overline{q}^{(n+2)}\,\dd x\\
        &\phantom{==}\pm 2\int \overline{q}^{(n)}\left[qC_+(q\overline{q})'\right]^{(n)} \pm q^{(n)}\left[\overline{q}C_-(q\overline{q})'\right]^{(n)}\dx.
    \end{align*}
    Notice that the first terms equals to $-2\varepsilon \|q^{(n+1)}\|^2_{L^2}\leq 0$ which is strictly non-positive, so we may only focusing on the nonlinearity now.  We assume inductively that $\left\| q\right\|_{H^{n-1}}\leq C$ $t\in[0,T]$ and $\eps\in(0,1]$. Without loss of generality, consider the term $$\pm 2\int \overline{q}^{(n)}\left[qC_+(q\overline{q})'\right]^{(n)}\,\dx,$$ and the last part will follow by exactly the same approach. Distributing the derivatives, we find
    \begin{multline*}
   2\int\overline{q}^{(n)}\left[qC_+(q\overline{q})'\right]^{(n)}\,\dx=\,2\int \overline{q}^{(n)}qC_+(q^{(n+1)}\overline{q})\,\dx\,+\,2\int \overline{q}^{(n)}qC_+(q\overline{q}^{(n+1)})\,\dx\,\\
   +\,2\int \overline{q}^{(n)}q^{(n)}C_+(q\overline{q})'\,\dx+\sum_{\substack{\alpha_1,\alpha_2,\alpha_3\geq 0\\ \alpha_1+\alpha_2+\alpha_3=n+1}}c_\alpha\int \overline{q}^{(n)}q^{(\alpha_1)}C_+(q^{(\alpha_2)}\overline{q}^{(\alpha_3)})\,\dx.
    \end{multline*}
Writing for short, we define
\begin{align*}
    I_1&\coloneqq \,2\int \overline{q}^{(n)}qC_+(q^{(n+1)}\overline{q})\,\dx,\\
    I_2&\coloneqq \,2\int \overline{q}^{(n)}qC_+(q\overline{q}^{(n+1)})\,\dx,\\
    I_3&\coloneqq \,2\int \overline{q}^{(n)}q^{(n)}C_+(q\overline{q})'\,\dx,\\
    I_4&\coloneqq \sum_{\substack{\alpha_1,\alpha_2,\alpha_3\geq 0\\ \alpha_1+\alpha_2+\alpha_3=n+1}}c_\alpha\int \overline{q}^{(n)}q^{(\alpha_1)}C_+(q^{(\alpha_2)}\overline{q}^{(\alpha_3)})\,\dx.
\end{align*}

For $I_1$, applying integration by parts once is suffices to reduce the order of $ q^{(n+1)} $. It follows that
{\small
   \begin{align*}
     \int q \overline{q}^{(n)} C_+ (q^{(n+1)} \overline{q}) \, \dx=-\int q\overline{q}^{(n+1)}C_+(q^{(n)}\overline{q})-\int \overline{q}^{(n)}q'C_+(q^{(n)}\overline{q})\,\dx-\int \overline{q}^{(n)}qC_+(q^{(n)}\overline{q}')\,\dx.
   \end{align*}}
Writing for short that
\begin{align*}
    \mathcal{A}&\coloneqq -\int q\overline{q}^{(n+1)}C_+(q^{(n)}\overline{q})\,\dx,\\
    \mathcal{B}&\coloneqq  -\int \overline{q}^{(n)}q'C_+(q^{(n)}\overline{q})\,\dx,\\
    \mathcal{C}&\coloneqq -\int \overline{q}^{(n)}qC_+(q^{(n)}\overline{q}')\,\dx.
\end{align*}
   Notice that $\mathcal{B}$ and $\mathcal{C}$ could be dealt by the same process as in $I_4$. For $A$, there is the same scenario happens when we cope with $$2\int q^{(n)}\left[\overline{q}C_-(q\overline{q})'\right]^{(n)}\,\dx,$$ therefore, it produces the term
   $$\int q^{(n)}\overline{q}C_-(q\overline{q}^{(n+1)})\,\dx.$$
   Combine with $\mathcal{A}$ we get
    \begin{align*}
        \int q^{(n)}\overline{q}C_-(q\overline{q}^{(n+1)})\,\dx+\mathcal{A}&= \int q^{(n)}\overline{q}C_-(q\overline{q}^{(n+1)})\,\dx-\int q\overline{q}^{(n+1)}C_+(q^{(n)}\overline{q})\,\dx\\
        &=\int C_+(q^{(n)}\overline{q})C_-(q\overline{q}^{(n+1)})\,\dx-\int C_-(q\overline{q}^{(n+1)})C_+(q^{(n)}\overline{q})\,\dx\\
        &=0.
    \end{align*}
  
   Term $I_2$ follows the approach for $I_1.$
   For $I_3$, it is following that there is another term exists symmetrically such as
$$I_{3_B}=\int q^{(n)}\overline{q}^{(n)}C_-(q'\overline{q})\,\dx.$$
Therefore, combine $I_3$ with this we obtain:
\begin{align*}
    \vert I_3+I_{3_\mathcal{B}}\vert&= \left \vert\int \overline{q}^{(n)}q^{(n)}C_+(q'\overline{q})\,\dx+\int q^{(n)}\overline{q}^{(n)}C_-(q'\overline{q})\,\dx\right\vert\\
    &=\left\vert \int \overline{q}^{(n)}q^{(n)}q'\overline{q}\,\dx\right\vert\\
    &\leq \left\| q^{(n)}\right\|_{L^2}^2\left\| q'\right\|_{L^\infty}\|\overline{q}\|_{L^\infty}\\
    &\lesssim_{\left\| q\right\|_{H^{n-1}}}\left\| q\right\|^2_{H^n}+1.
\end{align*}

For $I_4$, we have two possibilities as follows:
\begin{itemize}
    \item $\alpha_1\leq n-2$ and ($\alpha\leq n$, $\alpha_3\leq n-2$) or ($\alpha_2\leq n-2,$ $\alpha_3\leq n$);
    \item $\alpha_1=n-1$ and $\alpha_2+\alpha_3=2$.
\end{itemize}
We start from the first possibility, in this case,
\begin{align*}
    \left|\int \overline{q}^{(n)}q^{(\alpha_1)}C_+(q^{(\alpha_2)}\overline{q}^{(\alpha_3)})\right|&\leq \| \overline{q}^{(n)}q^{(\alpha_1)}\|_{L^2}\|C_+(q^{(\alpha_2)}\overline{q}^{(\alpha_3)})\|_{L^2}.
\end{align*}
We place $q^{(\alpha_1)}$ and the term with the lower derivative order, $\min\{\alpha_2,\alpha_3\}$ in the $L^\infty$ norm. Without loss of generality, say $\alpha_2<\alpha_3$, by Sobolev embedding and Gagliardo--Nirenberg inequality it follows that 
\begin{align*}
    \left|\int \overline{q}^{(n)}q^{(\alpha_1)}C_+(q^{(\alpha_2)}\overline{q}^{(\alpha_3)})\right|&\leq \| \overline{q}^{(n)}q^{(\alpha_1)}\|_{L^2}\|C_+(q^{(\alpha_2)}\overline{q}^{(\alpha_3)})\|_{L^2}\\
    &\leq \|\overline{q}^{(n)}\|_{L^2}\|\overline{q}^{(\alpha_3)}\|_{L^2}\|q^{(\alpha_1)}\|_{L^\infty}\|q^{(\alpha_2)}\|_{L^\infty}\\
    &\leq \left(\left\| q^{(n)}\right\|^2_{L^2}+\left\| q^{(n)}\right\|_{L^2}\left\| q\right\|_{H^{n-1}}\right)\left\| q\right\|^2_{H^{n-1}}\\
    &\lesssim \left(\left\| q^{(n)}\right\|^2_{L^2}+1\right)\left(\left\| q\right\|^2_{H^{n-1}}+\left\| q\right\|^4_{H^{n-1}}\right).
\end{align*}
Next, we consider the second possibility, in this case
\begin{align*}
    \left|\int \overline{q}^{(n)}q^{(\alpha_1)}C_+(q^{(\alpha_2)}\overline{q}^{(\alpha_3)})\right|&\leq \| \overline{q}^{(n)}q^{(\alpha_1)}\|_{L^2}\|C_+(q^{(\alpha_2)}\overline{q}^{(\alpha_3)})\|_{L^2}.
\end{align*}
Place the term $\overline{q}^{(n)}$ and $\max\{\alpha_2,\alpha_3\}$ in $L^2$ norm and everything else goes to $L^\infty$ norm. Without loss of generality, say $\max\{\alpha_2,\alpha_3\}=\alpha_2$. Therefore, $\alpha_3\leq 1$ and it follows that
\begin{align*}
    \left|\int \overline{q}^{(n)}q^{(\alpha_1)}C_+(q^{(\alpha_2)}\overline{q}^{(\alpha_3)})\right|&\leq \| \overline{q}^{(n)}q^{(\alpha_1)}\|_{L^2}\|C_+(q^{(\alpha_2)}\overline{q}^{(\alpha_3)})\|_{L^2}\\
    &\leq \|\overline{q}^{(n)}\|_{L^2}\|q^{(n-1)}\|_{L^\infty}\|q^{(\alpha_2)}\|_{L^2}\|\overline{q}^{(\alpha_3)}\|_{L^\infty}\\
    &\lesssim \|\overline{q}^{(n)}\|_{L^2}\|q^{(n-1)}\|_{H^1}\left\| q\right\|_{H^2}\|q^{(\alpha_3)}\|_{H^1}\\
    &\lesssim \left\| q^{(n)}\right\|^2_{L^2}\left\| q\right\|^2_{H^2}\\
    &\lesssim \left(\left\| q^{(n)}\right\|^2_{L^2}+1\right)\left(\left\| q\right\|^2_{H^{n-1}}+\left\| q\right\|^4_{H^{n-1}}\right).
\end{align*}
Therefore, we conclude $$\frac{d}{dt}\left(\left\| q^{(n)}\right\|^2_{L^2}\right)\leq\left(\left\| q^{(n)}\right\|^2_{L^2}+1\right) C(\left\| q\right\|_{H^{n-1}}).$$
Since this is almost conservation law and by induction $\left\| q\right\|_{H^{n-1}}<C$, by Grönwall inequality it follows that:
$$\snorm{q(t)}^2_{\dH^n}\leq \left(\|q^{(n)}(0)\|^2_{L^2}+1\right)e^{ Ct}-1.$$
\end{proof}
\begin{corollary}\label{CCM_eps_gwp}
    Fix an integer $s\geq 2$.  Then \eqref{CCM2} is globally well-posed in $H^s_+$, with $\|q\|_{L^2}^2 < 2\pi$ in the focusing case.
\end{corollary}
\begin{proof}
By Proposition \ref{CCM_eps_lwp} and the \textit{a-priori} estimates obtained in the previous proposition, fix $A>0$ and consider initial data $q(0)\in H^3_+$.  
We claim that the corresponding solution exists globally in time in $H^s_+$ for all integers $s\geq 3$.  

To see this, fix an arbitrary time horizon $T>0$. From the uniform \textit{a-priori} bounds in the previous proposition, there exists
$$
R = R\bigl(T,\|q(0)\|_{H^s_+}\bigr) > 0
$$
such that
$$
\sup_{t\in[0,T]}\snorm{q(t)}_{H^s_+} \leq R .
$$
Now, by Proposition \ref{CCM_eps_lwp}, there exists a local existence time $T_0\ll 1$, depending only on $R$, such that whenever $\|q(t_0)\|_{H^s_+}\leq R$ at some time $t_0$, the solution can be extended to the interval $[t_0,t_0+T_0]$ while remaining in $B_R(0)$.  

Thus the local solution can be iterated with lifespan $T_0\ll 1$ to cover $[0,T]$ in finitely many steps, which yields the desired global well-posedness.  
\end{proof}

%-----------------------------------------------------------------------------------------------------------------%
\section{Convergence of the approximations}
Let $g\in H^s$ with $s\geq 2$, and fix $\varepsilon\in(0,1]$. 
We define a regularized version of $g$ by
\begin{equation}\label{smooth_g}
    g_\varepsilon := P_{\leq \varepsilon^{-\frac{1}{3}}}g.
\end{equation}
By construction, $g_\varepsilon\in H^\infty$, and hence there exists a unique smooth solution $q_\varepsilon(x,t)=q(x,t;\varepsilon)$ with 
$q_\varepsilon\in C([0,T];H^r_+)$ for any $r\geq 3$. Next, we prove that $q_\eps$ is a Cauchy sequence in $C_tH^s_+$.  Say $w=q_\varepsilon-q_\delta$ where $\delta\leq \varepsilon$, and to prove $q_\eps$ is Cauchy in $C_tH^s_+$, it suffices to show that $w\to 0$ in $C_tH^s_+$ topology as $\eps\to 0$.
It is clear that $w$ satisfies the partial differential equation:
\begin{equation}
iw_t= (-1 + i\varepsilon) w'' + i(\varepsilon - \delta) (q_\delta)'' \pm 2i \left[ w C_{+}(|q_\varepsilon|^2)' + q_\delta C_{+}\left( (w\overline{w} + w\overline{q_\delta} + q_\delta\overline{w})' \right) \right].
\end{equation}
Here is a simple result giving bounds on various norms of $q_\varepsilon$ and $q_\delta$.
\begin{lemma}\label{convergence_bound}
Let $g\in H^s$ for some $s\geq 2$ and let $g_\varepsilon$ be a regularized version of $g$ defined in \eqref{smooth_g}, then as $\varepsilon \to 0$ the following bounds hold,
\begin{equation}
    \begin{array}{rcll}
\left\|g_{\eps}\right\|_{s+k} & =&\mathcal{O}\left(\eps^{- \frac{k}{3}}\right) \quad &\text { for } \quad k=1,2, \ldots \\
\left\|g-g_{\eps}\right\|_{s-k} & =&o\left(\varepsilon^{\frac{k}{3}}\right) \quad &\text { for } \quad k=1,2, \ldots \\
\left\|g-g_{\eps}\right\|_{s} & =&o(1) &
\end{array}
\end{equation} 
and the first bound holds uniformly on bounded subsets $H^s$ whereas the last two bounds hold uniformly for compact subsets of $H^s$.
\end{lemma}

\begin{proof}
    The argument follows from a direct computation of the Sobolev norm but in frequency space. It follows that
    \begin{align*}
        \left\|g_\eps\right\|^2_{H^{s+k}}&=\int_{\xi>0}|\widehat{g_\eps}|^2\,\left(1+|\xi|\right)^{2(s+k)}\ \dd\xi\\
        &\leq \sup_{|\xi|\leq \eps^{-1}}(1+|\xi|)^{2k}\int (1+|\xi|)^{2s}|\widehat{g_\eps}|^2\dd\xi\\
        &\lesssim \eps^{-\frac{2k}{3}}\left\|g_\eps\right\|^2_{H^s}.
    \end{align*}
    
    For the third part, we will argue by using Lebesgue Dominated Convergence theorem. It follows that 
    \begin{align*}
        \left\|g-g_\eps\right\|_{H^s}^2=\int_{\xi>0}\psi^2(\eps^{-\frac13}\xi)(1+|\xi|)^{2s}|\widehat{g}|^2\,\dd\xi,
    \end{align*}
    where $1-\psi$ is the symbol for $P_{-\eps^{-\frac13}}$.
    which is bounded. By Lebesgue Dominated Convergence theorem it follows that $$\left\|g-g_\eps\right\|_{H^s}^2\to 0\quad \text{as}\quad \eps\to 0.$$ 
    It remains to show if $g_n\to g$ in $H^s$ then implies $\left\|g^n_\eps-g_\eps\right\|_{H^s}\to 0$ as $\eps\to 0$ uniformly for $n=1,2,3,\dots.$ To show this, fix $\eta>0$ and for all $n$ it follows that
    \begin{align*}
        \left\|g^n_\eps-g_\eps\right\|_{H^s}^2&=\left\|(g^n-g)_\eps\right\|_{H^s}^2\\
        &\leq \int_{\xi>0} (1+|\xi|)^{2s}|\widehat{g^n}-\widehat{g}|^2\,\dd\xi\\
        &=\left\|g^n-g\right\|^2_{H^s}.
    \end{align*}
    Let a sufficiently large $N\in \mathbb{Z}$ be given so that if $n\geq N$ we have $\left\|g^n-g\right\|_{H^s}<\frac{\eta}{3}$. Then, fix a $\eps_0$ sufficiently small so that $\left\|g^j_\eps-g^j\right\|_{H^s}<\frac{\eta}{3}$ for $1\leq j\leq N$ and $\left\|g_\eps-g\right\|_{H^s}<\frac{\eta}{3}$ for a $\epsilon$ so small and $\eps\in(0,\epsilon]$. Thus, by triangle inequality we have $\left\|g^n_\eps-g^n\right\|_{H^s}<\eta$ and it holds for all $n$.

    For the second part, we have 
    \begin{align*}
        \|g-g_\varepsilon\|_{H^{s-k}}^2&=\int_{\xi>0} \frac{(1+|\xi|)^{2(s-k)}}{(1+|\xi|)^{2s}}\ \psi^2(\varepsilon\xi)(1+|\xi|)^{2s}|\widehat{g}|^2\dd\xi\\
        &\le \Big(\sup_{\xi\geq \eps^{-1}}(1+|\xi|)^{-2k}\Big)\ \int_{\xi>0} \psi(\varepsilon\xi)\,(1+|\xi|)^{2s}|\hat g(\xi)|^2\,\dd\xi\\
        &\le C\varepsilon^{\frac{2k}{3}} \int_{\xi>0} \psi(\varepsilon\xi)^2\,(1+|\xi|)^{2s}|\hat g(\xi)|^2\,\dd\xi
    \end{align*}
     where $C$ is a constant that does not depend on $g$ and $\eps$. Since $\|g-g_\eps\|_{H^s}=o(1)$ it implies that $$\|g-g_\varepsilon\|_{H^{s-k}}^{2}\le\big(\mathcal{O}(\varepsilon^{\frac{2k}{3}})\big)\cdot \big(o(1)\big)
= o(\varepsilon^{\frac{2k}{3}}).$$
\end{proof}

We next prove that $\{q_\eps \}$ is Cauchy in $L^2$ topology.
\begin{proposition}\label{aug7_0917}
Let $s\ge 2$ and $g\in H^s_+(\R)$. For $\varepsilon\in(0,1]$ set $g_\varepsilon:=P_{\le \varepsilon^{-\frac13}}g\in H^\infty_+$ and let $q_\varepsilon\in C([0,T];H^s_+)$ be the corresponding solution to \eqref{smooth_g}. Then, for any $T>0$, we have
$$
\sup_{t\in[0,T]}\|q_\varepsilon(t)-q_\delta(t)\|_{L^2}
\ \lesssim_T\,\varepsilon,
$$
uniformly in $0<\delta\le\varepsilon\le1$ and compact sets of $g\in H^s_+$.
In particular $\{q_\varepsilon\}$ is Cauchy in $C([0,T];L^2_+)$.
\end{proposition}
\begin{proof}
We are proving $w(t)\to 0$ in $L^2$.  Define
$$\mathscr{E}_0=\int \vert w(t,x)\vert^2\dx.$$
It follows that 
\begin{align*}
    \dddt \mathscr{E}_0&=2\Re\int w_t\overline{w}\dx\\
    &=2\Re\int \big\{(i+\varepsilon)w''+(\varepsilon-\delta)(q_\delta)''\\
    &\phantom{==}\pm 2\left[wC_+(\vert q_\varepsilon\vert^2)'+q_\delta C_+((\vert w\vert^2+w\overline{q_\delta}+\overline{w}q_\delta)')\right]\big\}\cdot \overline{w}\dx.
\end{align*}
Writing for short, we define
\begin{align*}
    I_1&=2\Re \int (i+\varepsilon)w''\cdot \overline{w}\dx,\\
    I_2&=2\Re\int (\varepsilon-\delta)(q_\delta)''\cdot \overline{w}\dx,\\
    I_3&=\pm 4\Re\int\left[wC_+(\vert q_\varepsilon\vert^2)'\right]\cdot \overline{w}\dx,\\
    I_4&=\pm 4 \Re\int \left[q_\delta C_+((\vert w\vert^2+w\overline{q_\delta}+\overline{w}q_\delta)')\right]\cdot \overline{w}\dx.
\end{align*}
\noindent
{\it{\underline{Estimate for $I_1$}}}.
We have, employing integration by parts once,
\begin{align*}
    I_1&=-2(i+\varepsilon)\int w'\overline{w}'\dx\\
    &=-2\varepsilon\left\| w\right\|_{\dH^1}^2\\
    &\leq 0.
\end{align*}
\noindent
{\it{\underline{Estimate for $I_2$}}}.
Using the \textit{a-priori} estimate in Proposition \ref{july23_0719}
and Young's inequality, it follows that
\begin{align*}
    \vert I_2\vert&\leq 2(\varepsilon-\delta)\left|\Re\int q''_\delta \overline{w}\dx\right|\\
    &\leq 2(\varepsilon-\delta)\|q''_{\delta}\|_{L^2}\left\| w\right\|_{L^2}\\
    &\lesssim \eps \left\| w\right\|_{L^2}.
\end{align*}
\noindent
{\it{\underline{Estimate for $I_3$}}}.
By Lemma \ref{convergence_bound} and the boundness of $C_+$ in $H^1$, it follows that
\begin{align*}
    \vert I_3\vert&\lesssim \left|\int \vert w\vert^2C_+(\vert q_\eps\vert^2)'\dx\right|\\
    &\leq \left\| w\right\|^2_{L^2}\|C_+(\vert q_\eps\vert^2)'\|_{L^\infty}\\
    &\lesssim \mathscr{E}_0\cdot \|C_+(\vert q_\eps\vert^2)'\|_{H^1}\\
    &\lesssim \mathscr{E}_0\cdot \|q_\eps\|^2_{H^2}\\
    &\lesssim \mathscr{E}_0.
\end{align*}
\noindent
{\it{\underline{Estimate for $I_4$}}}.
Define
\begin{align*}
|I_4| &\leq 4 \left| \Re\int q_\delta C_{+}((|w|^2)') \overline{w} \dx \right|+ 4 \left| \Re\int q_\delta C_{+}((w\overline{q}_\delta)') \overline{w} \dx \right| + 4 \left| \Re\int q_\delta C_{+}((\overline{w}q_\delta)') \overline{w}  \dx \right|\\
&\eqqcolon \vert\mathcal{A}\vert+\vert\mathcal{B}\vert+\vert\mathcal{C}\vert.
\end{align*}
For term $\mathcal{A}$, it follows that
\begin{align*}
\vert \mathcal{A}\vert &= \left| \int q_\delta C_{+}((|w|^2)') \overline{w} \dx \right| \\
&\leq  \|q_\delta\|_{H^1}(\|q_\eps\|_{H^2}+\|q_{\delta}\|_{H^2})\left\| w\right\|_{L^2}^2\\
&\lesssim \left\| w\right\|^2_{L^2}.
\end{align*}
For term $\mathcal{B}$, we have
$$\vert \mathcal B\vert= 4 \left| \Re\int q_\delta C_{+}(w\overline{q}_\delta') \overline{w} \dx \right|+ 4 \left| \Re\int q_\delta C_{+}(w'\overline{q}_\delta)\overline{w} \dx \right|\eqqcolon \vert\mathcal{B}_1\vert+\vert\mathcal{B}_2\vert.$$
For term $\mathcal{B}_1$, we have
\begin{align*}
    \vert \mathcal{B}_2\vert&\lesssim \left\| w\right\|^2_{L^2}\|q_\delta\|_{L^\infty}\|q_\delta'\|_{L^\infty}\\
    &\lesssim \mathscr{E}_0.
\end{align*}
For term $\mathcal{B}_2$, using the Lemma \ref{Commutator_type} and by Lemma \ref{convergence_bound}, we have 
\begin{align}\label{convergence_L2_I4_B}
    \vert \mathcal{B}_2\vert
    &\lesssim \left\| w\right\|^2_{L^2}\|q_\delta\|_{L^\infty}\|q_\delta'\|_{L^\infty}\lesssim \mathscr{E}_0.
\end{align}
Term $\mathcal{C}$ follows the estimate of \eqref{convergence_L2_I4_B}.
Thus, 
\begin{equation}
    \vert \frac{\dd}{dt} \mathscr{E}_0\vert\lesssim 2\varepsilon\left\| w\right\|_{\dH^1}^2+\varepsilon \sqrt{\mathscr{E}_0}+4\mathscr{E}_0\lesssim \mathscr{E}_0+\varepsilon^2.
\end{equation}
By Grönwall, it follows that
$$\mathscr{E}_0(t)\leq \mathscr{E}_0(0)e^{Ct}+C\varepsilon^2(e^{Ct}-1),$$
the right hand side of above goes to $0$ as $\varepsilon \to 0.$ Also, by Lemma \ref{convergence_bound} we acquire that
$$\mathscr{E}_0(0)=\|g_\varepsilon-g_\delta\|^2_{L^2}\lesssim \varepsilon^{2}.$$ 
Therefore, 
\begin{align}\label{initial_data_bound}
    \|w(t)\|_{L^2}=\sqrt{\mathscr{E}_0(t)}&\leq C_T(\sqrt{\mathscr{E}_0(0)}+\varepsilon)\nonumber\\
    &\lesssim_T \varepsilon.
\end{align}
\end{proof}
We have shown that the solution sequence $\{q_\eps\}$ to \eqref{CCM2} is Cauchy in $L^2$ class, we next start proving that $\{q_\eps\}$ is Cauchy in $H^s$ topology for $s\geq 2$.

\begin{proposition}\label{t:H2 conv}
Let $n\ge 3$ and $g\in H^n_+(\R)$ and $g_\eps$ as the definition in \eqref{smooth_g}. Then, for any $T>0$, we have
$$
\sup_{t\in[0,T]}\|q_\varepsilon(t)-q_\delta(t)\|_{H^n}
\ \lesssim_T\left\|q_\eps(0)-q_\delta(0)\right\|_{H^n}+\eps^\frac16,
$$
uniformly in $0<\delta\le\varepsilon\le1$ and compact sets of $g\in H^n_+$.
In particular $\{q_\varepsilon\}$ is Cauchy in $C([0,T];H^n_+)$.
\end{proposition}

\begin{proof}
Consider the quantity 
$$ \mathscr{E}_{2n}(t) = \left\| w(t)\right\|_{\dH^n}^2 . $$ 
We compute
\begin{align*}
    \frac{\dd}{dt}\mathscr{E}_{2n}&=2\Re \int w^{(n)}_t\overline{w}^{(n)}\,\dx\\
    &=2\Re\int w_t \overline{w}^{(2n)}\,\dx\\
    &=-2\Re \int \big\{(i+\varepsilon)w''+(\varepsilon-\delta)(q_\delta)''\\
    &\phantom{==}\pm 2\left[wC_+(\vert q_\varepsilon\vert^2)'+q_\delta C_+((\vert w\vert^2+w\overline{q_\delta}+\overline{w}q_\delta)')\right]\big\}\cdot \overline{w}^{(2n)}\,\dx.
\end{align*}
Writing for short, we define
\begin{align*}
    I_1&=2\Re \int (i+\varepsilon)w''\cdot \overline{w}^{(2n)}\dx,\\
    I_2&=2\Re\int (\varepsilon-\delta)(q_\delta)''\cdot \overline{w}^{(2n)}\dx,\\
    I_3&=\pm 4\Re\int\left[wC_+(\vert q_\varepsilon\vert^2)'\right]\cdot \overline{w}^{(2n)}\dx,\\
    I_4&=\pm 4 \Re\int \left[q_\delta C_+((\vert w\vert^2+w\overline{q_\delta}+\overline{w}q_\delta)')\right]\cdot \overline{w}^{(2n)}\dx.
\end{align*}
The estimate for $I_1$ and $I_2$ follows the same trick in the previous section, however, starting from term $I_3$ and $I_4$ will be slightly different.\\
\par \noindent
{\it{\underline{Estimate for $I_1$}}}.
We have
\begin{align*}
    I_1=-2\varepsilon\Re\int \vert w^{(n)}\vert^2\dx=-2\varepsilon \left\| w\right\|^2_{\dH^{n+1}}\leq 0.
\end{align*}
\noindent
{\it{\underline{Estimate for $I_2$}}}.
Using the \textit{a-priori} estimate in Proposition \ref{july23_0719}
and Young's inequality, it follows that
\begin{align*}
    \vert I_2\vert&\leq 2(\varepsilon-\delta)\left|\Re\int q''_\delta \overline{w}^{(2n)}\dx\right|\\
    &\leq 2(\varepsilon-\delta)\|q''_{\delta}\|_{L^2}\|w^{(2n)}\|_{L^2}\\
    &\lesssim \eps \left\| w\right\|_{\dH^{2n}}.
\end{align*}
{\it{\underline{Estimate for $I_3$}}}.
We have
$$\vert I_3\vert\lesssim \left\|w^{(n)}\right\|_{L^2}\left(\left\|w^{(n)}\right\|_{L^2}\left\|\vert q_\varepsilon\vert^2\right\|_{H^2}+\sum^{n-1}_{m=0}\left\|w^{(m)}\right\|_{L^\infty}\left\|(\vert q_\varepsilon\vert^2)^{n+1-m}\right\|_{L^2}\right)\eqqcolon \mathcal{A}+\mathcal{B}.$$
For term $\mathcal{A}$, we have
\begin{align*}
    \vert \mathcal{A}\vert&\lesssim \left\|w^{(n)}\right\|_{L^2}\left(\int (1+\vert \xi\vert^{2n})\vert \hat{w}\vert^2\ \dd \xi\right)^{\frac{1}{2}}\ \|\vert q_\varepsilon\vert^2\|_{H^2}\\
    &\lesssim \mathscr{E}_{2n}+\varepsilon^{2}.
\end{align*}
For term $\mathcal{B}$, by Gagliardo--Nirenberg inequality it yields
\begin{align*}
    \vert \mathcal{B}\vert&\lesssim \sqrt{\mathscr{E}_{2n}}\cdot \sum^{n-1}_{m=0}\left\|w^{(m)}\right\|_{L^\infty}\left\|\left(\vert q_\varepsilon\vert^2\right)^{(n+1-m)}\right\|_{L^2}\\
    &\lesssim \sqrt{\mathscr{E}_{2n}}\cdot \sum^{n-1}_{m=0}\left(\left\| w\right\|^{a_m}_{L^2}\cdot \left\| w\right\|^{b_m}_{\dH^n}\right)\cdot \left\|q_\eps\right\|^2_{H^{n+1-m}}\\
    &\lesssim \left(\varepsilon^{a_m}\cdot \mathscr{E}_{2n}^{\gamma_m}\right)\cdot \|q_\eps\|^2_{H^{n+1-m}}
\end{align*}
where
\begin{equation}\label{aug6_1546}
\begin{aligned}
    a_m&:= \frac{2(n-m)-1}{2n},\\
    b_m&:= \frac{2m+1}{2n},\\
    \gamma_m&:= \frac{1}{2}+\frac{b_m}{2}=\frac{2(n+m)+1}{4n},\\
    1-\gamma_m&= \frac{2(n-m)-1}{4n}= \frac{a_m}{2}.
\end{aligned}
\end{equation}
We will use these exponents repeatedly in the following estimates. Meanwhile, by Proposition 5.2 and 5.3 it follows that for any fixed $T>0$ there exists a constant $C_T$ that is independent from $\eps$ such that 
$\sup_{0<\eps\leq 1}\snorm{q_\eps}_{C([0,T],H^s_+)}\leq C_T$. This means that $q_\eps$ is bounded in $C_tH^s$ class.
When $m=0$, we estimate
\begin{align*}
    \vert \mathcal{B}\vert&\lesssim \left(\varepsilon^{a_0}\cdot \mathscr{E}_{2n}^{\gamma_0}\right)\cdot   \left(\varepsilon^{-\frac{1}{3}}\right)^2\\
    &\lesssim \varepsilon^{\frac{2n-3}{6n}}\cdot \mathscr{E}_{2n}^{\frac{2n + 1}{4n}} \\
    &\lesssim \mathscr{E}_{2n} + \varepsilon^{\frac{2(2n-3)}{3(2n-1)}}
\end{align*}
where the exponent $\frac{2(2n-3)}{3(2n-1)}$ is strictly positive, or more precisely, with the minimum value $\frac{2}{5}$ when $n\geq 3$.
For $m\geq 1$, we estimate
\begin{align*}
\vert \mathcal{B}\vert
    &\lesssim \sum^{n-1}_{m=1}\left(\varepsilon^{a_m}\cdot \mathscr{E}_{2n}^{\gamma_m}\right)\cdot  1\\
    &\lesssim \mathscr{E}_{2n} +\varepsilon^{\frac{a_m}{1-\gamma_m}}\\
    &\lesssim \mathscr{E}_{2n} +\varepsilon^2.
\end{align*}
Collecting the previous two steps, we conclude
$$\vert \mathcal{B}\vert\lesssim \mathscr{E}_{2n} + \varepsilon^{\frac{2}{5}}.$$
{\it{\underline{Estimate for $I_4$}}}.
Writing for short, define
\begin{align*}
    \mathcal{A}&\coloneqq \pm 4\Re \int q_\delta \overline{w}^{(2n)}C_+(\vert w\vert^2)'\dx,\\
    \mathcal{B}&\coloneqq \pm 4\Re \int q_\delta \overline{w}^{(2n)}C_+(w\overline{q_\delta})'\dx,\\
    \mathcal{C}&\coloneqq \pm 4\Re \int q_\delta \overline{w}^{(2n)}C_+(\overline{w}q_\delta)'\dx.
\end{align*}
For $\mathcal{A}$, employing integration by parts, it follows that
$$\vert\mathcal{A}\vert\lesssim \left\|w^{(n)}\right\|_{L^2}\left(\left\|q_\delta^{(n)}\right\|_{L^\infty}\left\|(\vert w \vert^2)'\right\|_{L^2}+\sum^{n-1}_{m=0}\left\|q_\delta^{(m)}\right\|_{L^\infty}\left\|(\vert w\vert^2)^{(n+1-m)}\right\|_{L^2}\right)\eqqcolon \mathcal{A}_1+\mathcal{A}_2.$$
We first estimate $\mathcal{A}_1$, by Gagliardo--Nirenberg it follows that
\begin{align*}
    \vert \mathcal{A}_1\vert&\lesssim \sqrt{\mathscr{E}_{2n}}\left\|q_\delta\right\|_{\dH^{(n+1)}}\left\|w\|_{L^2}\|w'\right\|_{L^\infty}\\
    &\lesssim \sqrt{\mathscr{E}_{2n}}\cdot {\mathscr{E}_{2n}}^{\frac{3}{4n}}\cdot 1\cdot \varepsilon^{\frac{2n-3}{2n}-\frac13+1}\\
    &\lesssim \eps^{\frac{10n-9}{6n}}\mathscr{E}_{2n}^{\frac{2n+3}{4n}}.
\end{align*}
Notice that here we are arguing the inductive step with $n\geq 3$, so the exponent $\frac{10n-9}{4n}$ is strictly positive. Hence, by Young's inequality
\begin{align*}
    \vert \mathcal{A}_1\vert &\lesssim \mathscr{E}_{2n}+\eps^{\frac{2(10n-9)}{3(2n-3)}}.
\end{align*}
With $n\geq 3$, the exponent $\frac{2(10n-9)}{3(2n-3)}\geq 2$, therefore, we reduce the estimate to 
$$\vert \mathcal{A}_1\vert \lesssim \mathscr{E}_{2n}+\eps^2.$$

For term $\mathcal{A}_2$, we simply use Gagliardo--Nirenberg inequality to interpolate the $w$ part into the appropriate space since by previous \textit{a-priori} estimate, the $\|q^{(m)}\|_{L^\infty}$ part is bounded follows by 
$$\|q^{(m)}_\delta\|_{L^\infty}\lesssim \left\|q_\delta\right\|_{H^{(m+1)}}.$$  
We first discuss the case that $m\neq 0$, the strategy is to using product rule and assign the higher derivative terms into $L^2$ and following by using Gagliardo--Nirenberg interpolation, the term with the lower derivatives goes to $L^\infty$ with bounded using \textit{a-priori} estimate. 
Define
$$\mathcal{S}^{\geq}_m=\{(\alpha,\beta)\in \mathbb{N}^2_0: \ \alpha+\beta=n+1-m,\ \alpha \geq \beta\}$$
and without loss of generality, we only discuss the case that $\alpha\geq \beta.$  By Gagliardo--Nireberg interpolation, \textit{a-priori} estimate, and Young's inequality we have
\begin{align*}
    \vert \mathcal{A}_2\vert &\lesssim \left\|w^{(n)}\right\|_{L^2}\left(\sum^{n-1}_{m=0}\left\|q_\delta^{(m)}\right\|_{L^\infty}\left(\sum_{(\alpha,\beta)\in\mathcal{S}^{\geq}_m}\left\| w^{(\alpha)}\right\|_{L^2}\left\|\overline{w}^{(\beta)}\right\|_{L^\infty}\right)\right)\\
    &\lesssim \sqrt{\mathscr{E}_{2n}}\cdot 1\cdot\left(\left\|w\right\|_{L^2}^{\frac{n-\alpha}{n}} \left\|w\right\|_{\dot{H}^n}^{\frac{\alpha}{n}}\cdot \left\|w^{(\beta)}\right\|_{L^\infty}\right)\\
    &\lesssim \left(\varepsilon^{\frac{n-\alpha}{n}} \mathscr{E}_{2n}^{\frac{\alpha+n}{2n}}\right)\cdot\left(\left\|q_\delta\right\|_{\dH^{\beta+1}}+\left\|q_\varepsilon\right\|_{\dH^{\beta+1}}\right)\\
    &\lesssim \mathscr{E}_{2n}+\varepsilon^2.
\end{align*}
Notice that here $\frac{n - \alpha}{n} > 0$ and $0 < \frac{\alpha}{n} \leq 1$ for $\alpha \in [0, n + 1 - m]$, under the constraint $n \geq 3$.
For case $m=0$, we only discuss the case where $\alpha=n+1$ and $\beta=0$ or $\alpha=0$ and $\beta=n+1$ since these are the only two cases we need to use Lemma \ref{convergence_bound}. Without loss of generality, we consider the case $\alpha=n+1$ and $\beta=0$ here,
\begin{align*}
    \vert \mathcal{A}_2\vert&\lesssim \sqrt{\mathscr{E}_{2n}}\cdot 1\cdot\left\|(\vert w\vert^2)^{n+1}\right\|_{L^2} \lesssim \sqrt{\mathscr{E}_{2n}}\cdot\left(\sum_{\mathcal{S}^{>}_0} \left\|w^{(\alpha)}\overline{w}^{(\beta)}\right\|_{L^2}\right),
\end{align*}
where
$$\mathcal{S}^{>}_0=\{(\alpha,\beta)\in \mathbb{N}^2_0: \ \alpha+\beta=n+1,\ \alpha > \beta\}.$$
By Lemma \ref{convergence_bound} it follows that
\begin{align*}
     \vert \mathcal{A}_2\vert &\lesssim \sqrt{\mathscr{E}_{2n}}\cdot 1\cdot\left(\left\|w^{(n+1)}\overline{w}\right\|_{L^2}\right)\\
     &\lesssim \sqrt{\mathscr{E}_{2n}}\cdot \left(\left\|w\right\|_{H^{n+2}}\left\|w\right\|_{L^2}\right)\\
     &\lesssim \sqrt{\mathscr{E}_{2n}}\cdot (\varepsilon^{-\frac{2}{3}}\cdot \varepsilon)\\
     &\lesssim \mathscr{E}_{2n} + \varepsilon^{\frac{1}{3}}.
\end{align*}

Also, notice that when $0<\alpha\leq n$, the estimate could be finished by using the Gagliardo--Nirenberg to do interpolation and the proof follows from the case $m\neq 0$ in $\mathcal{A}_2$.
We next estimate the term $\mathcal{B}$. It follows that
$$\vert \mathcal{B}\vert\lesssim \left\|w^{(n)}\right\|_{L^2}\left(\left\|q_\delta^{(n)}\right\|_{L^2}\left\|(w\overline{q_\delta})'\right\|_{L^\infty}+\sum^{n-1}_{m=0}\left\|q_\delta^{(m)}\right\|_{L^\infty}\left\|(w\overline{q_\delta})^{(n+1-m)}\right\|_{L^2}\right)\eqqcolon \mathcal{B}_1+\mathcal{B}_2.$$
For $\mathcal{B}_1$, it follows that
\begin{align*}
    \vert \mathcal{B}_1\vert&\lesssim \sqrt{\mathscr{E}_{2n}}\cdot 1\cdot \left(\left\|w\right\|_{H^2}\|q_\delta\|_{H^1}+\left\|w\right\|_{H^1}\|q_\delta\|_{H^2}\right)\\
    &\lesssim \sqrt{\mathscr{E}_{2n}}\cdot \big(\|w\|^{\frac{n-2}{n}}_{L^2}\|w\|^{\frac{1}{n}}_{\dH^n}\cdot\|q_\delta\|_{H^1}+\left\|w\right\|^{\frac{n-1}{n}}_{L^2}\left\|w\right\|^{\frac{1}{2n}}_{\dH^n}\|q_\delta\|_{H^2}\big)\\
    &\lesssim \mathscr{E}_{2n} + \varepsilon^2.
\end{align*}
For $\mathcal{B}_2$, there would be two cases. We first discuss the case when $m\neq 0$. Define
$$\mathcal{S}_m=\{(\alpha,\beta)\in \mathbb{N}^2_0: \ \alpha+\beta=n+1-m\}.$$
By the Gagliardo--Nirenberg interpolation, \textit{a-priori} estimate, and Young's inequality it follows that
\begin{align*}
    \vert \mathcal{B}_2\vert &\lesssim \left\|w^{(n)}\right\|_{L^2}\left(\sum^{n-1}_{m=0}\left\|q_\delta^{(m)}\right\|_{L^\infty}\left(\sum_{(\alpha,\beta)\in \mathcal{S}_m}\left\|w^{(\alpha)}\right\|_{L^2}\left\|q_\delta\right\|_{\dH^{\beta+1}}\right)\right)\\
    &\lesssim \sqrt{\mathscr{E}_{2n}}\cdot 1\cdot \left\|w\right\|_{L^2}^{\frac{n-\alpha}{n}} \left\|w\right\|_{\dot{H}^n}^{\frac{\alpha}{n}}\\
    &\lesssim \varepsilon^{\frac{n-\alpha}{n}}\cdot \mathscr{E}_{2n}^{\frac{\alpha+n}{2n}}\\
    &\lesssim  \mathscr{E}_{2n}+\varepsilon^2.
\end{align*}
It remains to estimate the case where $m=0$. When all derivatives hit on $q_\delta$, we have the estimate
\begin{align*}
    \vert \mathcal{B}_2\vert &\lesssim \left\|w^{(n)}\right\|_{L^2}\left(\sum^{n-1}_{m=0}\left\|q_\delta^{(m)}\right\|_{L^\infty}\left\|w\right\|_{L^2}\left\|q_\delta\right\|_{\dH^{n+2}}\right)\\
    &\lesssim \sqrt{\mathscr{E}_{2n}}\cdot \varepsilon\cdot \eps^{-\frac23}\\
    &\lesssim \mathscr{E}_{2n}+\varepsilon^{\frac13}.
\end{align*}
When all derivatives hit on $w$, by Lemma \ref{Commutator_type}, it follows that
\begin{align*}
    \vert \mathcal{B}_2\vert &= \int \overline{w}^{(n)}\,q_\delta'\,C_+\!\big(q_\delta\,\overline{w}^{(n)}\big)\dx
  + \int \overline{w}^{(n)}\,q_\delta\,C_+\!\big(q_\delta'\,\overline{w}^{(n)}\big)\dx\\
  &\lesssim 2\left\|w^{(n)}\right\|_{L^2}\left\|q'_\delta\right\|_{L^\infty}\left\|q_\delta\right\|_{L^\infty}\left\|w^{(n)}\right\|\\
  &\lesssim \mathscr{E}_{2n}.
\end{align*}
The estimate of term $\mathcal{C}$ follows from the estimate of term $\mathcal{B}$. Summing over all terms give
$$\frac{\dd}{dt}\mathscr{E}_{2n}\lesssim C\mathscr{E}_{2n}+\varepsilon^{\frac13},$$
with $K$ the least coefficient of $1$ in all factors $\varepsilon$ occurring above estimate.
By Grönwall, it follows that
$$\ \mathscr{E}_{2n}(t)\ \le\ e^{Ct}\,\mathscr{E}_{2n}(0)\ +\ \frac{e^{Ct}-1}{C}\,\varepsilon^{\frac13},\qquad 0\le t\le T.$$
Taking square roots and combining with the $L^2$-estimate in Proposition \ref{aug7_0917}, we obtain
$$\sup_{t\in[0,T]}\|w(t)\|_{H^n}\lesssim_T\|w(0)\|_{H^n}+\varepsilon^{\frac16}.$$
\end{proof}
Having established the strong convergence of the regularized solutions in the corresponding Sobolev topologies, we are now positioned to state the main result of this section which is the global existence and uniqueness of solutions to the \eqref{CCM}. This is the first piece of Theorem \ref{main_thm}.
\begin{theorem}\label{aug14_1435}
Fix an integer $s \geq 3$. For any initial data $g \in H^s_+(\mathbb{R})$, there exists a solution $q(t)$ for the initial-value problem of \eqref{CCM} belonging to $C([0,T];H^s_+)$ for any $T>0$. Moreover, such a solution is unique in the class $C([0,T];H^s_+)$.
\end{theorem}
\begin{proof}
    Let $g\in H^s$ be given and set $q_\eps(0)=P_{\leq \eps^{-\frac13}}g\in H^\infty$, by the previous section 
    there exists a $q_\eps(t)$ that solves \eqref{CCM2} which also converges. Set $q(t)=\lim_{\eps\to 0}q_\eps(t)$. By noticing that if $s>\frac{3}{2}$, then $H^{s-1}$ is a Banach algebra. Hence, it follows that as $\eps\to 0$, 
    \begin{equation*}
        \begin{cases}
            \partial^2_xq_\eps\to \partial^2_x q\quad \text{in}\quad C([0,T]; H^{s-2}_+),\\
            q_\eps C_+(|q_\eps|^2)'\to qC_+(|q|^2)'\quad \text{in}\quad  C([0,T]; H^{s-1}_+).
        \end{cases}
    \end{equation*}
    Also, by \textit{a-priori} estimate, it follows that 
    $$\|q_\eps\|_{C([0,T]; H^s_+)}\leq C_T,$$
    so that 
    $$\|\eps \partial^2_xq_\eps\|_{C([0,T]); H^{s-2}_+}\leq \eps C_T\to 0\quad \operatorname{as}\quad \eps\to 0.$$
    We thus proved that \eqref{CCM2} converges to \eqref{CCM} and so the $q(t)$ solves the initial-value problem for \eqref{CCM}.

    For uniqueness result, fix $q_u$ and $q_v$ as two solutions of \eqref{CCM}, and define $q_w=q_u-q_v$. Since $q_u$ and $q_v$ are continuous from $[0,T]$ into $H^s_+$, therefore, $q_u$ and $q_v$ is bounded in $H^s_+$. Hence, $q_w$ satisfies
\begin{equation}\label{unique_w}
    \frac{\dd}{dt}q_w=i q_w^{\prime \prime} \pm\left(2 q_w C_{+}\left(|q_u|^{2}\right)^{\prime}+2 q_v C_{+}\left[\left(|q_w|^{2}+q_w \overline{q_v}+q_v \overline{q_w}\right)^{\prime}\right]\right).
\end{equation}
The $L^2$ energy gives
\begin{align*}
    \frac{\dd}{dt}\left\|q_w(t)\right\|^2_{L^2}&=2\Re\int (q_w)_t\overline{q_w}\dx\\
    &=2\Re\int \left\{i q_w^{\prime \prime} \pm\left(2 q_w C_{+}\left(|q_u|^{2}\right)^{\prime}+2 q_v C_{+}\left[\left(|q_w|^{2}+q_w \overline{q_v}+q_v \overline{q_w}\right)^{\prime}\right]\right)\right\}\overline{q_w}\dx.
\end{align*}
Linear dispersive term follows that
$$\Re\int iq''_w\overline{q_w}\dx=0,$$
and it remains to estimate the nonlinear terms. For the first part of the nonlinear terms, we have the estimate
\begin{align*}
    \left|\Re\int q_wC_+(|q_u|^2)'\overline{q_w}\dx\right|&\leq C\left\|q_w\right\|_{L^2}^2
\end{align*}
that $C$ depends on $\snorm{q_u}_{H^s}$.
The second part, we have
\begin{align*}
    \Re \int  q_v\overline{q_w} C_{+}\left[\left(|q_w|^{2}+q_w \overline{q_v}+q_v \overline{q_w}\right)^{\prime}\right]\dx \eqqcolon \mathcal{A}+\mathcal{B}+\mathcal{C}.
\end{align*}
For $\mathcal{A}$, it follows that
\begin{align*}
    \vert \mathcal{A}\vert&\lesssim \left\|q_v\right\|_{L^\infty}\left\|q_w\right\|_{L^2}^2\left\|q_w'\right\|_{L^\infty}\\
    &\leq C\cdot \left\|q_w\right\|_{L^2}^2\left(\left\|q_u\right\|_{H^2}+\left\|q_v\right\|_{H^2}\right)\\
    &\leq C\cdot \left\|q_w\right\|_{L^2}^2.
\end{align*}
We next assign the $\mathcal{B}$ into two parts and estimate them differently,
$$\mathcal{B}=\Re\int q_v\overline{q_w}C_+(q_w'\overline{q_v})+ q_v\overline{q_w}C_+(q_w\overline{q_v}')\dx\eqqcolon \mathcal{B}_1+\mathcal{B}_2.$$
For $\mathcal{B}_1$, by Lemma \ref{Commutator_type}, we have
\begin{align*}
    \vert \mathcal{B}_1\vert &\lesssim \int q_v'\overline{q_w}C_+(q_w\overline{q_v})\dx+\int q_v\overline{q_w}C_+(q_w\overline{q_v}')\dx\\
    &\leq C\left\|q_w\right\|^2_{L^2}.
\end{align*}
For $\mathcal{B}_2$, it is a part of $\mathcal{B}_1$ after we used the Lemma \ref{Commutator_type}, it follows that
\begin{align*}
    \vert \mathcal{B}_2\vert&\lesssim \left\|q_v\right\|_{L^\infty}\left\|q_w\right\|_{L^2}^2\left\|q_v'\right\|_{L^\infty}\\
    &\leq C\left\|q_w\right\|^2_{L^2}.
\end{align*}
Summing up over all pieces it gives
$$\frac{\dd}{dt}\left\|q_w(t)\right\|^2_{L^2}\leq C\left\|q_w(t)\right\|^2_{L^2},$$
and thus, by Grönwall we have
$$\left\|q_w(t)\right\|^2_{L^2}\leq e^{Ct}\left\|q_w(0)\right\|^2_{L^2}=0\quad \forall t\geq 0.$$
Hence $q_w$ is $0$ everywhere. 
\end{proof}
%--------------------------------------------------------------------------------------------------------------------------%
\section{Continuous dependence of solutions on the initial data}
In this section, we are proving the last piece of elements for global well-posedness, which is the Theorem as follows. In view of Theorem 5.4, this will finish the proof of Theorem \ref{main_thm}.
\begin{theorem}[Continuous dependence on the initial data]\label{thm:cont-dep}
Fix $T>0$ and an integer $s\ge 3$. Let $\Phi:H^s_+ \to C([0,T];H^s_+)$ denote the solution map of \eqref{CCM},
which assigns to $g\in H^s_+$ the unique solution $q=\Phi(g)$ on $[0,T]$ constructed in Theorem \ref{aug14_1435}. 
Then $\Phi$ is continuous as a map $H^s_+\to C([0,T];H^s_+)$.
Moreover, the map
$$
  S: H^3_+ \longrightarrow C([0,T];H^{s-2}_+),\qquad S(g):=\partial_t\Phi(g),
$$
is continuous.
\end{theorem}
\begin{proof}
Our first goal is to prove that $\Phi$ is a continuous mapping. Say, $q(0),\tilde{q}(0)\in H^s_+$ and $q(t)=\Phi(q(0))$, $\tilde{q}(t)=\Phi(\tilde{q}(0))$ are the associated solutions of \eqref{CCM} initial value problems posed with initial data $q(0),\tilde{q}(0)$ respectively, then we have
\begin{align*}
    \left\|q_t-\tilde{q}_t\right\|_{H^{s-2}}&\leq \left\|q''\pm 2qC_+(\vert q\vert^2)'-\tilde{q}''\pm 2\tilde{q}C_+(\vert \tilde{q}\vert^2)'\right\|_{H^{s-2}}\\
    &\leq \left\|q''-\tilde{q}''\right\|_{H^{s-2}}+ \left\| (q-\tilde{q})C_+(\vert q\vert^2)'\right\|_{H^{s-2}}+\left\|\tilde{q}\left[C_+(\vert q\vert^2)'-C_+(\vert \tilde{q}\vert^2)'\right]\right\|_{H^{s-2}}\\
    &\leq I_1+I_2+I_3.
\end{align*}

For $I_1$, we estimate
\begin{align*}
    \left\|(q-\tilde{q})''\right\|_{H^{s-2}}&\lesssim \|q-\tilde{q}\|_{H^s}.
\end{align*}

For $I_2$, by noticing that $H^{s-2}$ is an algebra since $s\geq 3$ we have
\begin{align*}
    \left\| (q-\tilde{q})C_+(\vert q\vert^2)'\right\|_{H^{s-2}}&\lesssim \|q-\tilde{q}\|_{H^{s-2}}\|q\|^2_{H^{s-2}}\\
    &\lesssim \left\|q-\tilde{q}\right\|_{H^{s}}.
\end{align*}

For $I_3$, by noticing that $H^{s-1}$ is an algebra since $s\geq 3$ we have
\begin{align*}
    \left\|\tilde{q}\left[C_+(\vert q\vert^2)'-C_+(\vert \tilde{q}\vert^2)'\right]\right\|_{H^{s-2}}&\lesssim \|\tilde{q}\|_{H^{s-2}}\left(\|q-\tilde{q}\|_{H^{s-1}}\|q\|_{H^{s-1}}+\|q-\tilde{q}\|_{H^{s-1}}\|\tilde{q}\|_{H^{s-1}}\right)\\
    &\lesssim \snorm{q-\tilde{q}}_{H^s}
\end{align*}
Taking supremum over $t\in[0,T]$ yields
$$\left\|S(q(0))-S(\tilde{q}(0))\right\|_{C([0,T]; H^{s-2})}=\left\|q_t-\tilde{q}_t\right\|_{C([0,T]; H^{s-2})}\lesssim \left\|\Phi(q(0))-\Phi(\tilde{q}(0))\right\|_{C([0,T]; H^s_+)}$$
which shows that the continuity of $S:H^s_+\to C([0,T]; H^{s-2}_+)$ will follow from the continuity of $\Phi:H^s\to C([0,T]; H^s_+)$.
\\\\
\noindent
To show $\Phi:H^s_+\to C([0,T]; H^s_+)$ is continuous, let $q^n(0)\to q(0)$ in $H^s_+$ where $s\geq 3$, fix $q^n=\Phi(q^n(0))$ and $q=\Phi(q(0))$ be the associated solutions of \eqref{CCM}. Denote $q_\eps$, $q_\eps^n$ is the solution of \eqref{CCM2} with the smoothed data $q(0)_\eps$ and $q_\eps^n(0)$, respectively. Following by the \textit{a-priori} estimate Theorem \ref{jul29_1403}
we have $\|q\|_{H^1}\lesssim 1$, and also $\left\|q_\eps-q_\delta\right\|_{H^{s-1}}\lesssim \eps$.
By Proposition~\ref{t:H2 conv}, we have
$$\left\|q_\delta^n-q_\eps^n\right\|_{H^s}\lesssim \eps^{\frac16}+(\left\|q^n(0)-q_\eps^n(0)\right\|_{H^s}+\left\|q^n(0)-q_\delta^n(0)\right\|_{H^s})$$
uniformly for $n\geq 1$.
Sending $\delta\to 0$, we have $q_\delta\to q$ in $C([0,T];H^s)$ and so it follows that
\begin{align*}
    \left\|q^n-q_\eps^n\right\|_{C([0,T]; H^s_+)}\lesssim \eps^{\frac16}+\|q^n(0)-q_\eps^n(0)\|_{H^s_+}
\end{align*}
uniformly for $n\geq 1$, and similarly
\begin{equation*}
    \|q-q_\eps\|_{C([0,T]; H^s_+)}\lesssim \varepsilon^{\frac16}+\|q(0)-q_\eps(0)\|_{H^s_+} .
\end{equation*}
%Since $q^n(0)$ converges in $H^s_+$, there is a constant $N>0$ such that $\|q^n(0)\|_{H^3_+}\leq N$ for all $n$. 
Therefore, we conclude that as $\eps\to 0$ we have 
$$\sup_{n\geq 1}\|q^n-q_\eps^n\|_{C([0,T]; H^s_+)}\to 0$$
and 
$$\|q-q_\eps\|_{C([0,T]; H^s_+)}\to 0.$$

Fix $\eta>0$.  By the previous paragraph, there exists a $\varepsilon$ such that $\|q^n-q_\eps^n\|_{C([0,T]; H^s_+)}\leq \frac{\eta}{3}$ for all $n\geq 1$ and $\|q_\eps-q\|_{C([0,T]; H^s_+)}\leq \frac{\eta}{3}$. For this $\eps$, we know that $\|q_\eps^n-q_\eps\|_{C([0,T]; H^s_+)}\to 0$ as $n\to\infty$, by corollary \ref{CCM_eps_gwp} we have the global well-posedness for \eqref{CCM2}. Thus, there exists a $M>0$ such that $\|q_\eps^n-q_\eps\|_{C([0,T]; H^s_+)}\leq \frac{\eta}{3}$ for all $n\geq M$. Thus, it follows that 
$$\left\|q-q^n\right\|_{C([0,T]; H^s_+)}\leq \left\|q-q_\eps\right\|_{C([0,T]; H^s_+)}+\left\|q_\eps-q_\eps^n\right\|_{C([0,T]; H^s_+)}+\left\|q_\eps^n-q^n\right\|_{C([0,T]; H^s_+)}\leq \eta.$$
The proof of the theorem is now complete.
\end{proof}

% \bib, bibdiv, biblist are defined by the amsrefs package.

\end{document}